\documentclass[a4paper] {article}
[12pt]

\makeatletter
\renewcommand*\l@section{\@dottedtocline{1}{1.5em}{2.3em}}
\makeatother

\usepackage{amsfonts}
\usepackage{amssymb}
\usepackage[T1]{fontenc}

\usepackage{tikz}\usepackage{float}
\usetikzlibrary{calc}

\usepackage{CJK}
\usepackage{amsmath}
\usepackage{subfigure}
\usepackage{amsfonts}
\usepackage{amssymb}
\usepackage{amsthm}
\usepackage{amssymb}
\usepackage{enumerate}
\usepackage[calc]{picture}
\usepackage[all,cmtip]{xy}
\usepackage{diagbox}
\usepackage{makecell}
\usepackage[mathscr]{eucal}
\usepackage{eqlist}

\usepackage{color}
\usepackage{abstract}
\usepackage[T1]{fontenc}

 \usepackage[top=3.0 cm, bottom=3.0cm, left=3.5cm, right= 3.5cm]{geometry}

\usepackage{cite}

\theoremstyle{plain}
\newtheorem{theorem}{Theorem}[section]
\newtheorem{proposition}[theorem]{Proposition}
\newtheorem{lemma}[theorem]{Lemma}
\newtheorem{corollary}[theorem]{Corollary}

\theoremstyle{definition}
\newtheorem{definition}{Definition}[section]

\newtheorem{example}{Example}[section]

\theoremstyle{remark}
 \newtheorem{remark}{Remark}[section]

\usepackage{etoolbox}

\numberwithin{equation}{section}
\numberwithin{theorem}{section}

\begin{document}
    \begin{CJK*}{GBK}{kai}
    \CJKtilde

\begin{center}
{\Large {\textbf {The magnitude homology of a hypergraph
}}}
 \vspace{0.58cm}

Wanying Bi, Jingyan Li, Jie Wu

\bigskip

\bigskip

    \parbox{24cc}{{\small
{\textbf{Abstract}. The magnitude homology, introduced by R. Hepworth and S. Willerton, offers a topological invariant that enables the study of graph properties. Hypergraphs, being a generalization of graphs, serve as popular mathematical models for data with higher-order structures. In this paper, we focus on describing the topological characteristics of hypergraphs by considering their magnitude homology. We begin by examining the distances between hyperedges in a hypergraph and establish the magnitude homology of hypergraphs. Additionally, we explore the relationship between the magnitude and the magnitude homology of hypergraphs. Furthermore, we derive several functorial properties of the magnitude homology for hypergraphs. Lastly, we present the K\"{u}nneth theorem for the simple magnitude homology of hypergraphs.}

}}
\end{center}

\vspace{1cc}



\section{Introduction}
In everyday life, objects come in various sizes, and the same applies to mathematical objects, which also possess different sizes. For example, the size of a set is determined by its cardinality, subsets of $\mathbb{R}^n$  have volumes, vector spaces have dimensions, and topological spaces have Euler characteristics, among other examples. We can observe that these sizes satisfy the following two equations:
\begin{gather*}
  {\rm Size}(A\cup B)= {\rm Size}(A)+ {\rm Size}(B)- {\rm Size}(A\cap B) ,\\
   {\rm Size}(A\times B)=  {\rm Size}(A)\times {\rm Size}(B).
\end{gather*}
Let's consider whether it is possible to define such a size in a metric space. Tom Leinster, a topologist, was inspired by the Euler characteristic in topology and aimed to extend its application to metric spaces. Acknowledging the distinctions between metric spaces and topological spaces, he introduced a novel invariant called ``magnitude'' to describe the structure and complexity of metric spaces\cite{leinster2013magnitude}. The vision of ``magnitude'' was first introduced by T. Leinster in \cite{leinster2013magnitude}, which is analogous to the Euler characteristic of a category \cite{leinster2008euler}. Throughout the course of biological development, magnitude has made its presence known. A. R. Solow and S. Polasky introduced an invariant called ``effective number of species'' and is identical to magnitude \cite{solow1994measuring}. In 2009, T. Leinster stated and proved a new maximum entropy theorem, which also showed that in a steady state, the magnitude can be considered to be the maximum diversity, close to the maximum entropy\cite{leinster2009maximum}.

Now, let's consider the case of topological spaces $X$, where we often use the Euler characteristic to describe the size of a given space. However, the Euler characteristic is a relatively coarse invariant. An important improvement to measure the size of a space is the concept of ordinary homology, which is an algebraic invariant defined by a sequence of abelian groups $H_n(X)$. The Euler characteristic can be expressed as the alternating sum of the ranks of these homology groups. T. Leinster has indeed introduced the concept of magnitude, which determines the size of finite metric spaces. This prompts us to consider the possibility of developing a homology theory specifically for magnitude in finite metric spaces. Such a theory would capture additional geometric information beyond numerical magnitude, with magnitude being defined as the alternating sum of the ranks of the magnitude homology groups \cite{leinster2021magnitude,hepworth2017categorifying}. Next, we will focus on introducing the magnitude homology of graphs.

After defining the magnitude of finite metric spaces, T. Leinster firstly extended his consideration to the special case of graphs as metric spaces. Graphs, as a specific type of metric space where the distance between vertices is measured by the length of the shortest path, also serve as a framework for the development of the theory of magnitude \cite{leinster2019magnitude}. Subsequently, R. Hepworth and S. Willerton propose the magnitude homology of a graph as a categorification of magnitude \cite{hepworth2017categorifying}. After this paper, research on the magnitude homology of graphs gradually began. It is clear from \cite{hepworth2017categorifying} that the calculation of the magnitude homology of a graph is complex and difficult, hence the emergence of several technical methods for the calculation of the magnitude homology of graphs \cite{gu2018graph,sazdanovic2021torsion,asao2021geometric,bottinelli2021magnitude,asao2020geometric}. Some of the questions proposed by R. Hepworth and S. Willerton \cite{hepworth2017categorifying} are also addressed in these papers . For example, in 2018, Y. Gu answered in Appendix A of \cite{gu2018graph} whether graphs with the same magnitude and different magnitude homology exist; R. Sazdanovic and V. Summers conducted an analysis of the structure and implications of torsion in magnitude homology\cite{sazdanovic2021torsion}. Numerous other studies on magnitude homology in graphs are also currently underway.
In \cite{asao2021girth}, the magnitude homology of the graph is investigated concerning the relationship between its diagonality and girth. Y. Tajima and M. Yoshinaga investigate the relationship between the homotopy  type of the CW complex and the diagonality of magnitude homology groups \cite{tajima2021magnitude}. In 2023, Y. Asao constructed a spectral sequence whose first page is isomorphic to magnitude homology  \cite{asao2023magnitude}.

Higher-order interactions in complex networks are one of the most challenging scientific problems, and a great deal of research has been carried out based on this subject. In 2020, F. Battiston et al. gave a complete overview of the burgeoning field of networks and further discussed how to represent higher-order interactions \cite{battiston2020networks}. Some works on different frameworks for describing higher-order systems are also presented. Two models that are significant in higher-order interaction networks are also mentioned, that is, simplicial complexes and hypergraphs \cite{battiston2020networks}. Simplicial complexes and hypergraphs are  mathematical frameworks that can explicitly and naturally describe group interactions. Even though simplicial complexes solve some of the problems encountered with other low-dimensional representations, they are still limited by the requirement that all subfaces exist. Hypergraphs, as the generalization of abstract simplicial complex (remove the restrictions required by simplicial complex), provide the most general and unconstrained description of higher-order interactions. In recent years, works modeled on hypergraphs in complex networks have been developing rapidly \cite{alvarez2021evolutionary,battiston2021physics,lotito2022higher,contisciani2022inference,aksoy2020hypernetwork}. In 2021, the authors not only generalize BI (bipartite implementation) to a fully higher-order case but also serve as a theoretical basis for the study of higher-order cooperative games in uniform and heterogeneous hypergraphs, while also illustrating the impact of higher-order interactions in the evolutionary process \cite{alvarez2021evolutionary}. Data scientists have found that higher-order interactions can occur in larger groups by using hypergraphs as a model\cite{battiston2021physics}.

Research on the magnitude homology of graphs is still ongoing. However, little attention has been given to exploring the application of magnitude homology in the context of hypergraphs, which are higher-order extensions of graphs. This research gap serves as the primary focus of our study, as we aim to bridge this knowledge gap and investigate it in our paper.

In this paper, we first define the distance of the simplices or hyperedges on complexes and hypergraphs. Through this construction, we can define the magnitude and magnitude homology of hypergraphs. Given a hypergraph $\mathcal{H}$, the relationship between its magnitude and magnitude homology is obtained (see Theorem \ref{Eluer}).
\begin{theorem}
Let $\mathcal{H}$ be a hypergraph. Then
\begin{equation*}
  \sum\limits_{n,l\geq 0}(-1)^{n}\mathrm{rank}(\mathbf{MH}_{n,l}(\mathcal{H}))\cdot q^{l}=\#\mathcal{H}.
\end{equation*}
Here, $n\in \mathbb{Z},l\in \frac{1}{2}\mathbb{Z}$.
\end{theorem}
\noindent The functorial properties of magnitude homology of hypergraphs are also considered.

We introduce the magnitude chain complex of hypergraphs by employing tuples composed of hyperedges. However, computing this magnitude directly is always a challenging task due to the potentially large number of hyperedges involved. To address this issue, we propose a simplified version known as the simple magnitude homology, which lends itself better to computational analysis. Furthermore, the simple magnitude homology of hypergraphs can be seen as a generalization of the magnitude homology of graphs. Additionally, the simple magnitude homology can present a K\"{u}nneth formula. Through the Cartesian product of hypergraphs, we obtain the exterior product of simple magnitude chain complexes of hypergraphs. This, in turn, induces the exterior product of simple magnitude homology
$$
\Box: \mathrm{MH}_{*,*}(\mathcal{G})\otimes \mathrm{MH}_{*,*}(\mathcal{H})\longrightarrow \mathrm{MH}_{*,*}(\mathcal{G}\Box \mathcal{H}).
$$
And then the K\"{u}nneth theorem is constructed (see Theorem \ref{kunneth}).
\begin{theorem}
\textbf{\emph{(The K\"{u}nneth theorem for simple magnitude homology of hypergraphs)}}
Let $\mathcal{G}$ and $\mathcal{H}$ be hypergraphs. By the exterior product, we have a natural short exact sequence
\begin{equation*}
\begin{split}
      0\rightarrow \bigoplus\limits_{p+q=n}\mathrm{MH}_{p,*}(\mathcal{G})\otimes \mathrm{MH}_{q,*}(\mathcal{H})&\stackrel{\square}{\rightarrow}\mathrm{MH}_{n,*}(\mathcal{G}\Box\mathcal{H})\\
      &\rightarrow \bigoplus\limits_{p+q=n}\mathrm{Tor}(\mathrm{MH}_{p,*}(\mathcal{G}),\mathrm{MH}_{q-1,*}(\mathcal{H}))\rightarrow 0.
\end{split}
\end{equation*}
\end{theorem}

The paper is structured as follows: In the next section, we introduce the concept of distance  of the simplices or hyperedges on complexes and hypergraphs. Using this definition, we construct the magnitude and magnitude homology of a hypergraph, which leads us to our first main result. Section \ref{section:functorial} explores the functorial properties of magnitude homology for hypergraphs. Finally, in Section \ref{section:kunneth}, we present the proof of the K\"{u}nneth theorem for the simple magnitude homology of hypergraphs.

\section{Magnitude homology of hypergraphs}
\subsection{The distance on complexes and hypergraphs}
Recall that the distance of two vertices $x,y$ of a graph $G$ is defined by the length of the shortest edge path from $x$ to $y$.

\begin{definition}
Let $K$ be a simplicial complex, and let $\sigma$, $\tau$ be two simplices in $K$. A \emph{path from $\sigma$ to $\tau$} is a sequence $\sigma_{0}\sigma_{1}\sigma_{2}\cdots\sigma_{k}$ of simplices in $K$ with $\sigma=\sigma_{0},\tau=\sigma_{k}$ such that every two connecting simplices have a nonempty intersection, i.e., $\sigma_{i}\cap \sigma_{i+1}\neq \emptyset$ for $i=0,\dots,k-1$.
\end{definition}

Now, we will introduce the length of a path on a simplicial complex\footnote{The distance defined in our paper follows from the idea in \cite{aksoy2020hypernetwork}, which is a crucial tool to help us generalize graph-based network science techniques to hypergraphs.}.
\begin{definition}\label{definition:length}
Let $K$ be a simplicial complex. Let $\gamma=\sigma_{0}\sigma_{1}\sigma_{2}\cdots\sigma_{k}$ be a path on $K$. The \emph{length of $\gamma$} is defined by $\ell(\gamma)=\sum\limits_{i=0}^{k-1} \ell(\sigma_i, \sigma_{i+1})$. Here, for $\sigma\cap \tau\neq \emptyset$,
\begin{equation*}
  \ell(\sigma, \tau)=\begin{cases}
0,&  \sigma=\tau;\\
\frac{1}{2},& \sigma\subsetneq \tau\ or\ \tau \subsetneq \sigma;\\
1, & {\rm otherwise}.
\end{cases}
\end{equation*}
We denote $h(\gamma)=k$, called the \emph{height of $\gamma$}.
\end{definition}

\begin{example}\label{ex:2.1}
Given a simplicial complex $\Delta[2]$=$\{\{0\}, \{1\}, \{2\}, \{0,1\}, \{0,2\}$, $\{1,2\}, \{0,1,2\}\}$, we calculate the length and height of the paths from $\{0\}$ to $\{2\}$ on it. Let us take the example of two paths from $\{0\}$ to $\{2\}$, $\gamma_1=\{0\}\{0,1\}\{1,2\}\{2\}$ and $\gamma_2=\{0\}\{0,1\}\{1\}\{1,2\}\{2\}$. By Definition \ref{definition:length}, we can obtain $\ell(\gamma_1)=\frac{1}{2}+1+\frac{1}{2}=2$, $\ell(\gamma_2)=\frac{1}{2}+\frac{1}{2}+\frac{1}{2}+\frac{1}{2}=2$, $h(\gamma_1)=3$ and $h(\gamma_2)=4$.
\end{example}

\begin{definition}\label{definintion: distance}
Let $\sigma,\tau$ be simplices of a simplicial complex $K$.
The \emph{intercrossing distance} between $\sigma$ and $\tau$ is defined by $d(\sigma,\tau)=\inf\limits_{\gamma}\ell(\gamma)$, where $\gamma$ runs across all the paths from $\sigma$ to $\tau$.
The \emph{external distance} between $\sigma$ and $\tau$ is defined by $\delta(\sigma,\tau)=\inf\limits_{\gamma}h(\gamma)$.
If there is no path from $\sigma$ to $\tau$, we denote $d(\sigma,\tau)=\delta(\sigma,\tau)=\infty$.
\end{definition}

In particular, if $K$ is a graph, the distance of simplices coincides with the distance of vertices. More precisely, we have
\begin{equation*}
  d(\sigma,\tau)=\delta(\sigma,\tau)-1=d_{G}(\sigma,\tau)
\end{equation*}
for 0-simplices $\sigma,\tau$ of $K$. Here, $d_{G}$ denotes the distance of vertices on a graph \cite{entringer1976distance}.

Recall that an \emph{extended metric} on a space $X$ is a function $d:X\times X\to \mathbb{R}\cup \{+\infty\}$ satisfying the positive definiteness, symmetry, and triangle inequality.
Let $K$ be a simplicial complex and let $X(K)$ be the set of simplices of $K$. Then the intercrossing distance and the external distance $d,\delta:X(K)\times X(K)\to \mathbb{R}\cup \{+\infty\}$ are extended metrics.

\begin{lemma}\label{lemma:path1}
Let $\sigma,\tau$ be simplices of a simplicial complex $K$. Then there exists a path $\gamma$ from $\sigma$ to $\tau$ such that $\ell(\gamma)=d(\sigma,\tau)$ and $h(\gamma)=\delta(\sigma,\tau)$.
\end{lemma}
\begin{proof}
Suppose $\gamma=\sigma_{0}\sigma_{1}\cdots\sigma_{n}$ is a path from $\sigma=\sigma_{0}$ to $\tau=\sigma_{n}$ of length $\ell(\gamma)=d(\sigma,\tau)$. If $\sigma_{i-1}\cap \sigma_{i+1}\neq \emptyset$ for some $1\leq i\leq n-1$, then $\gamma'=\sigma_{0}\sigma_{1}\cdots\sigma_{i-1}\sigma_{i+1}\cdots\sigma_{n}$ is also a path from $\sigma$ to $\tau$. Note that $\ell(\sigma_{i-1}, \sigma_{i+1})\leq \ell(\sigma_{i-1}, \sigma_{i})+\ell(\sigma_{i}, \sigma_{i+1})$. We have
\begin{equation*}
  \ell(\gamma')=\ell(\sigma_{i-1}, \sigma_{i+1})+\sum\limits_{j\neq i,i+1}\ell(\sigma_{j}, \sigma_{j+1})\leq \sum\limits_{j=0}^{k-1} \ell(\sigma_{j}, \sigma_{j+1})=\ell(\gamma).
\end{equation*}
Since $\ell(\gamma)=d(\sigma,\tau)$, one has $\ell(\gamma')=\ell(\gamma)$. Then $\gamma'$ is a path from $\sigma$ to $\tau$ of length $\ell(\gamma)=d(\sigma,\tau)$. By induction on the above progress, we can find a path from $\sigma$ to $\tau$ of length $\ell(\gamma)=d(\sigma,\tau)$ such that $\sigma_{i-1}\cap \sigma_{i+1}=\emptyset$ for all $1\leq i\leq n-1$. Moreover, it is impossible that $\sigma_{i}\subseteq \sigma_{i-1}$ or $\sigma_{i}\subseteq \sigma_{i+1}$ for some $1\leq i\leq n-1$. Indeed, if $\sigma_{i}\subseteq \sigma_{i-1}$, we have
\begin{equation*}
  \emptyset\neq\sigma_{i+1}\cap \sigma_{i}\subseteq\sigma_{i+1}\cap \sigma_{i-1},
\end{equation*}
it is a contradiction. Thus we obtain
\begin{equation*}
  \ell(\gamma)=\sum\limits_{j=0}^{n-1}\ell(\sigma_{j}, \sigma_{j+1})=\ell(\sigma_{0}, \sigma_{1})+\ell(\sigma_{n-1}, \sigma_{n})+n-2.
\end{equation*}
It follows that $n-1\leq\ell(\gamma)\leq n$. Thus we have $n-1\leq d(\sigma,\tau)\leq n$.

On the other hand, suppose $\tilde{\gamma}=\tilde{\sigma}_{0}\tilde{\sigma}_{1}\cdots\tilde{\sigma}_{k}$ is a path from $\sigma=\tilde{\sigma}_{0}$ to $\tau=\tilde{\sigma}_{k}$ of height $h(\tilde{\gamma})=k=\delta(\sigma,\tau)$. Then we have $\tilde{\sigma}_{i-1}\cap \tilde{\sigma}_{i+1}=\emptyset$ for all $1\leq i\leq k-1$. Otherwise, the path $\tilde{\gamma}$ can be reduced to a path with a height smaller than $k$, as shown in the previous construction. It follows that
\begin{equation*}
  \ell(\tilde{\gamma})=\sum\limits_{j=0}^{k-1}\ell(\tilde{\sigma}_{j}, \tilde{\sigma}_{j+1})=\ell(\tilde{\sigma}_{0}, \tilde{\sigma}_{1})+\ell(\tilde{\sigma}_{k-1}, \tilde{\sigma}_{k})+k-2.
\end{equation*}
So we have $\ell(\tilde{\gamma})\leq\delta(\sigma,\tau)\leq \ell(\tilde{\gamma})+1$.

If $n=\delta(\sigma,\tau)$, then $\gamma$ is the desired path. If $n\neq\delta(\sigma,\tau)$, then we have
\begin{equation*}
  \delta(\sigma,\tau)\leq n-1\leq d(\sigma,\tau)\leq\ell(\tilde{\gamma})\leq\delta(\sigma,\tau).
\end{equation*}
It follows that $\ell(\tilde{\gamma})=d(\sigma,\tau)$. Then $\tilde{\gamma}$ is the desired path.
\end{proof}
\begin{remark}
The proof of Lemma \ref{lemma:path1} also shows that
\begin{equation*}
  d(\sigma,\tau)\leq\delta(\sigma,\tau)\leq d(\sigma,\tau)+1.
\end{equation*}
\end{remark}

The following example is a particularly convincing illustration of Lemma \ref{lemma:path1}.
\begin{example}\label{ex:2.2}
Based on Example \ref{ex:2.1}, and considering all paths from  $\{0\}$ to $\{2\}$ on $\Delta[2]$, we will see that the two shortest paths are $\gamma_1'=\{0\}\{0,2\}\{2\}$ and $\gamma_2'=\{0\}\{0,1,2\}\{2\}$. By Definition \ref{definintion: distance}, the intercrossing distance is $d(\{0\},\{2\})=\ell(\gamma_1')=\ell(\gamma_2')=\frac{1}{2}+\frac{1}{2}=1$ and the external distance is $\delta(\{0\},\{2\})=h(\gamma_1')=h(\gamma_2')=2$.
\end{example}

\begin{lemma}\label{lemma:path2}
Let $K$ be a simplicial complex. Let $\sigma_{0}\sigma_{1}\cdots\sigma_{n}$ be a path on $K$ such that $\sigma_{i-1}\cap \sigma_{i+1}=\emptyset$ for all $1\leq i\leq n-1$. Then there exists a path $\sigma_{0}\sigma_{1}'\cdots\sigma_{n-1}'\sigma_{n}$ such that $\sigma_{1}',\dots,\sigma_{n-1}'\in K_{1}$.
\end{lemma}
\begin{proof}
We will complete the proof by the induction. When $n=1$, it is trivial. Considering the case $n=2$, then we have a path $\sigma_{0}\sigma_{1}\sigma_{2}$ from $\sigma_{0}$ to $\sigma_{2}$ such that $\sigma_{0}\cap\sigma_{2}=\emptyset$. By definition, we have $\sigma_{0}\cap \sigma_{1}\neq \emptyset$ and $\sigma_{1}\cap \sigma_{2}\neq \emptyset$. Then, there exist vertices $a\in\sigma\cap \sigma_{1}$ and $b\in\sigma_{1}\cap \sigma_2$. Note that $a\neq b$ since $\sigma_{0}\cap\sigma_{2}=\emptyset$.
So we have $\{a,b\}\subseteq \sigma_{1}$, the vertices $a$ and $b$ span a $1$-simplex $\sigma_{1}'$ such that $\sigma_{0}\cap \sigma_{1}'\neq \emptyset$ and $\sigma_{1}'\cap \sigma_{2}\neq \emptyset$. Since $\sigma_{1}$ is a simplex, we have $\sigma_{1}'\subseteq \sigma_{1}\in K_{\geq 1}$. It follows that $\sigma_{0}\sigma_{1}'\sigma_{2}$ is a path from $\sigma_{0}$ to $\sigma_{2}$ for $\sigma_{1}'\in K_{1}$.

Assume that the lemma is true for $n\leq m-1$. When $n=m$. Suppose that $\sigma_{0}\sigma_{1}\cdots\sigma_{m}$ is a path from $\sigma_{0}$ to $\sigma_{m}$ such that $\sigma_{i-1}\cap \sigma_{i+1}=\emptyset$ for all $1\leq i\leq m-1$.
Then $\sigma_{0}\sigma_{1}\cdots\sigma_{m-1}$ is a path from $\sigma_{0}$ to $\sigma_{m-1}$ such that $\sigma_{i-1}\cap \sigma_{i+1}=\emptyset$ for all $1\leq i\leq m-1$. By induction, there is a path $\sigma_{0}\sigma_{1}'\cdots\sigma_{m-2}'\sigma_{m-1}$ such that $\sigma_{1}',\dots,\sigma_{m-2}'\in K_{1}$. There is also a path $\sigma_{m-2}'\sigma_{m-1}\sigma_{m}$ such that $\sigma_{m-2}'\cap\sigma_{m}\subseteq \sigma_{m-2}\cap\sigma_{m}=\emptyset$. So there exists a 1-simplex $\sigma_{m-1}'$ such that $\sigma_{m-2}'\sigma_{m-1}'\sigma_{m}$ is a path on $K$. Thus we have a path $\sigma_{0}\sigma_{1}'\cdots\sigma_{m-1}'\sigma_{m}$ from $\sigma_{0}$ to $\sigma_{m}$ such that $\sigma_{1}',\dots,\sigma_{m-1}'\in K_1$. The proof is completed.
\end{proof}

\begin{proposition}\label{proposition:path}
Let $\sigma,\tau$ be simplices of a simplicial complex $K$.
Then there exists a path $\sigma\sigma_{1}\sigma_{2}\cdots\sigma_{n-1}\tau$ of length $l$ and height $\delta(\sigma,\tau)$ with $\sigma_{1},\dots,\sigma_{n-1}\in K_{1}$.
Here,
\begin{equation*}
  \ell=\left\{
                     \begin{array}{ll}
                       \delta(\sigma,\tau)-1=d(\sigma,\tau), & \hbox{$\sigma,\tau\in K_{0}$}; \\
                       \delta(\sigma,\tau), & \hbox{$\sigma,\tau\in K_{\geq1}$}; \\
                       \delta(\sigma,\tau)-\frac{1}{2}, & {\rm otherwise.}
                     \end{array}
                   \right.
\end{equation*}
\end{proposition}
\begin{proof}
By Lemma \ref{lemma:path1}, we have a path $\gamma=\sigma_{0}\sigma_{1}\cdots\sigma_{n}$ is a path from $\sigma=\sigma_{0}$ to $\tau=\sigma_{n}$ of length $\ell(\gamma)=d(\sigma,\tau)$ and height $h(\gamma)=n=\delta(\sigma,\tau)$.
If $\sigma_{i-1}\cap \sigma_{i+1}\neq \emptyset$ for some $1\leq i\leq n-1$, then $\gamma'=\sigma_{0}\sigma_{1}\cdots\sigma_{i-1}\sigma_{i+1}\cdots\sigma_{n}$ is also a path from $\sigma$ to $\tau$.
So we have $h(\gamma')=n-1<n=\delta(\sigma,\tau)$, contradiction. Thus we have $\sigma_{i-1}\cap \sigma_{i+1}= \emptyset$ for all $1\leq i\leq n-1$. It follows that
\begin{equation*}
  \ell(\gamma)=\ell(\sigma_{0}, \sigma_{1})+\ell(\sigma_{n-1}, \sigma_{n})+n-2.
\end{equation*}
By Lemma \ref{lemma:path2}, we have a path $\gamma_{1}=\sigma_{0}\sigma_{1}'\cdots\sigma_{n-1}'\sigma_{n}$ from $\sigma=\sigma_{0}$ to $\tau=\sigma_{n}$ such that $\sigma_{1}',\dots,\sigma_{n-1}'\in K_{1}$. Note that
\begin{equation*}
  \ell(\gamma_{1})=\ell(\sigma_{0}, \sigma_{1}')+\ell(\sigma_{n-1}', \sigma_{n})+n-2.
\end{equation*}

When $\sigma,\tau\in K_{0}$. Then we have $\ell(\sigma_{0}, \sigma_{1}')=\ell(\sigma_{n-1}', \sigma_{n})=\ell(\sigma_{0}, \sigma_{1})=\ell(\sigma_{n-1}, \sigma_{n})=1/2$. It follows that
\begin{equation*}
  d(\sigma,\tau)=\ell(\gamma)=\ell(\gamma_{1})=1+n-2=n-1.
\end{equation*}

When $\sigma,\tau\in K_{\geq 1}$. Recall that $\sigma_{1}'\not\subseteq \sigma_{0}=\sigma$ and $\sigma_{n-1}'\not\subseteq \sigma_{n}=\tau$. Since $\sigma_{1}'\in K_{1}$, it is impossible $\sigma\subseteq \sigma_{1}'$. Indeed, if $\sigma\subseteq \sigma_{1}'$, then $\sigma=\sigma_{1}'$ by the dimension reason. Then $\sigma\cap \sigma_{2}'=\sigma_{1}'\cap \sigma_{2}'\neq \emptyset$, which leads to a contradiction. So we have
\begin{equation*}
  \ell(\sigma, \sigma_{1}')=1.
\end{equation*}
Similarly, we obtain $\ell(\sigma_{n-1}', \tau)=1$. It follows that
\begin{equation*}
   \ell(\gamma_{1})=2+n-2=n.
\end{equation*}

When $\sigma\in K_{\geq 1}, \tau\in K_{0}$ or $\sigma\in K_{0}, \tau\in K_{\geq 1}$. We only consider the case $\sigma\in K_{\geq 1}, \tau\in K_{0}$ as the other case is similar. Since $\sigma\in K_{\geq 1}, \tau\in K_{0}$, we have $\ell(\sigma, \sigma_{1}')=1$ and $\ell(\sigma_{n-1}, \tau)=\ell(\sigma_{n-1}', \tau)=1/2$. It follows that
\begin{equation*}
  \ell(\gamma_{1})=1+\frac{1}{2}+n-2=n-\frac{1}{2}.
\end{equation*}
The desired result follows.
\end{proof}

Let $K$ be a simplicial complex. Then the $1$-skeleton $\mathrm{sk}_{1}(K)$ of $K$ is a graph. In view of Proposition \ref{proposition:path}, for vertices $x,y\in K$, we have
\begin{equation*}
  d_{\mathrm{sk}_{1}(K)}(x,y)=d_{K}(x,y).
\end{equation*}

Let $\mathcal{H}$ be a hypergraph, and let $\sigma,\tau$ be hyperedges in $\mathcal{H}$. A \emph{path from $\sigma$ to $\tau$ on $\mathcal{H}$} is a sequence of hyperedges $\sigma\sigma_{1}\sigma_{2}\cdots\sigma_{k}\tau$ of $\mathcal{H}$ such that every two connecting hyperedges have a nonempty intersection. The length and height of a path of hyperedges are defined in a similar way as Definition \ref{definition:length}.
\begin{definition}
The \emph{intercrossing distance} of the hyperedges $\sigma,\tau\in \mathcal{H}$ is defined by $d(\sigma,\tau)=\inf\limits_{\gamma}\ell(\gamma)$, where $\gamma$ runs across all the paths from $\sigma$ to $\tau$ on the hypergraph $\mathcal{H}$. The \emph{external distance} between $\sigma$ and $\tau$ is defined by $\delta(\sigma,\tau)=\inf\limits_{\gamma}h(\gamma)$.
We denote $d(\sigma,\tau)=\delta(\sigma,\tau)=\infty$ if there is no path from $\sigma$ to $\tau$.
\end{definition}

Let $\Delta \mathcal{H}$ be the \emph{simplicial closure} of $\mathcal{H}$, i.e., $\Delta \mathcal{H}=\{\tau\neq\emptyset|\tau\subseteq \sigma\text{ for some }\sigma\in\mathcal{H}\}$.
\begin{proposition}\label{proposition:path2}
Let $\mathcal{H}$ be a hypergraph. Then for hyperedges $\sigma,\tau\in \mathcal{H}$, we have
\begin{equation*}
  d_{\mathcal{H}}(\sigma,\tau)=d_{\Delta \mathcal{H}}(\sigma,\tau),\quad \delta_{\mathcal{H}}(\sigma,\tau)=\delta_{\Delta \mathcal{H}}(\sigma,\tau).
\end{equation*}
Moreover, there exists a path $\gamma$ on $\mathcal{H}$ from $\sigma$ to $\tau$ such that $\ell(\gamma)=d_{\mathcal{H}}(\sigma,\tau)$ and $h(\gamma)=\delta_{\mathcal{H}}(\sigma,\tau)$.
\end{proposition}
\begin{proof}
$(i)$ It is obviously that $d_{\Delta \mathcal{H}}(\sigma,\tau)\leq d_{\mathcal{H}}(\sigma,\tau)$ and $\delta_{\Delta \mathcal{H}}(\sigma,\tau)\leq \delta_{\mathcal{H}}(\sigma,\tau)$.

$(ii)$ We will prove $d_{\mathcal{H}}(\sigma,\tau)\leq d_{\Delta \mathcal{H}}(\sigma,\tau)$ and $\delta_{\mathcal{H}}(\sigma,\tau)\leq \delta_{\Delta \mathcal{H}}(\sigma,\tau)$. By Lemma \ref{lemma:path1}, we have a path $\gamma=\sigma_{0}\sigma_{1}\cdots\sigma_{n}$ on $\Delta \mathcal{H}$ from $\sigma=\sigma_{0}$ to $\tau=\sigma_{n}$ of length $\ell(\gamma)=d(\sigma,\tau)$ and height $h(\gamma)=n=\delta(\sigma,\tau)$. It is obvious for $n=1$. We will consider the case for $n\geq 2$.
By the definition of simplicial closure, we have a sequence of simplices $\sigma'_{1},\sigma'_{2},\dots,\sigma'_{n-1}$ of $\mathcal{H}$ such that $\sigma_{i}\subseteq\sigma'_{i}$ for $i=1,2,\dots,n-1$. Thus $\gamma'=\sigma\sigma'_{1}\cdots\sigma'_{n-1}\tau$ is a path from $\sigma$ to $\tau$ of height $h(\gamma')=n=\delta(\sigma,\tau)$. Similar to the proof of Proposition \ref{proposition:path}, we have
\begin{equation*}
   \ell(\gamma)=\ell(\sigma_{0}, \sigma_{1})+\ell(\sigma_{n-1}, \sigma_{n})+n-2.
\end{equation*}
If $\ell(\sigma_{0}, \sigma_{1})<\ell(\sigma_{0}, \sigma_{1}')$, we have $\ell(\sigma_{0}, \sigma_{1})=1/2$ and $\ell(\sigma_{0}, \sigma_{1}')=1$. Recall that $\sigma_{0}\cap\sigma_{2}=\emptyset$. So one has $\sigma_{0}\subsetneq\sigma_{1}$. It follows that $\sigma_{0}\subsetneq\sigma_{1}'$. We obtain $\ell(\sigma_{0}, \sigma_{1}')=1/2$, which leads to a contradiction. Thus $\ell(\sigma_{0}, \sigma_{1}')\leq \ell(\sigma_{0}, \sigma_{1})$. So we have
\begin{equation*}
  \ell(\gamma')\leq \ell(\sigma_{0}, \sigma'_{1})+\ell(\sigma_{n-1}', \sigma_{n})+n-2\leq \ell(\gamma),
\end{equation*}
which implies $d_{\mathcal{H}}(\sigma,\tau)\leq d_{\Delta \mathcal{H}}(\sigma,\tau)$. Thus $\gamma'$ is the desired path such that $\ell(\gamma')\leq d_{\Delta\mathcal{H}}(\sigma,\tau)$ and $h(\gamma')\leq \delta_{\Delta\mathcal{H}}(\sigma,\tau).$
\end{proof}

\begin{example}\label{ex:2.3}
Let $\mathcal{H}=\{\{0\},\{1\},\{2\},\{0,1\},\{1,2\},\{0,1,2\}\}$, and it is known that $\Delta \mathcal{H}=\Delta[2]$. Considering the elements $\{0\}$ and $\{2\}$ in $\mathcal{H}$ and $\Delta \mathcal{H}$, from Example \ref{ex:2.2}, we can obtain the shortest path from $\{0\}$ to $\{2\}$ on $\mathcal{H}$ and $\Delta \mathcal{H}$, i.e., $\gamma_2'=\{0\}\{0,1,2\}\{2\}$. By calculation, we have $d_{\mathcal{H}}(\{0\}, \{2\})=d_{\Delta \mathcal{H}}(\{0\}, \{2\})=\frac{1}{2}+\frac{1}{2}=1$ and $\delta_{\mathcal{H}}(\{0\}, \{2\})=\delta_{\Delta\mathcal{H}}(\{0\}, \{2\})=2$.
\end{example}

\subsection{The magnitude of a hypergraph}
Let $\mathbb{Z}[\sqrt{q}]$ be a polynomial ring over the integers in one variable $\sqrt{q}$.
For a finite hypergraph $\mathcal{H}$, let $Z_\mathcal{H} = Z_{\mathcal{H}}(q)$ be the square matrix over $\mathbb{Z}[\sqrt{q}]$ whose rows and columns are indexed by the hyperedges of $\mathcal{H}$, and whose $(\sigma, \tau)$-entry is
\begin{equation*}
  Z_{\mathcal{H}}(q)(\sigma, \tau) = q^{d(\sigma, \tau)},\quad \sigma, \tau\in \mathcal{H},
\end{equation*}
where by convention $q^{\infty} = 0$, $0<q<1$. Since $Z_{\mathcal{H}}(q)$ is invertible, let $ \sum (Z_{\mathcal{H}}(q))^{-1}$ denote the sum of all the elements in the inverse matrix.
\begin{definition}
\emph{The magnitude of the hypergraph} $\mathcal{H}$ is defined to be
\begin{equation*}
  \# \mathcal{H}(q)=\# \mathcal{H}= \sum (Z_{\mathcal{H}}(q))^{-1}\in \mathbb{Q}(\sqrt{q}).
\end{equation*}
Here, $\mathbb{Q}(\sqrt{q})$ is the quotient field of $\mathbb{Z}[\sqrt{q}]$.
\end{definition}

Note that $\# \mathcal{H}$ can be regarded as a formal power series over $\mathbb{Z}$, that is, $\#\mathcal{H}\in \mathbb{Z}[\![\sqrt{q}]\!]$.
Let $w_{\mathcal{H}}:\mathcal{H}\rightarrow \mathbb{Q}(\sqrt{q})$ be a function from the hyperedges to the field $\mathbb{Q}(\sqrt{q})$ given by
\begin{equation*}
  w_{\mathcal{H}}(\sigma)=\sum\limits_{\tau\in \mathcal{H}}(Z_{\mathcal{H}}(q))^{-1}(\sigma, \tau).
\end{equation*}
The function $w_{\mathcal{H}}$ is called the \emph{weighting on $\mathcal{H}$}.
Since $(Z_{\mathcal{H}}(q))(Z_{\mathcal{H}}(q))^{-1}=I_{|\mathcal{H}|}$, we have the \emph{weighting equation}
\begin{equation}
  \sum\limits_{\sigma\in \mathcal{H}}q^{d(\sigma,\tau)}w_{\mathcal{H}}(\sigma)=1,\quad \tau\in \mathcal{H}. \label{(3)}
\end{equation}
Here, $|\mathcal{H}|$ denotes the number of hyperedges of $\mathcal{H}$.

\begin{lemma}\label{lemma:weight}
Let $\mathcal{H}$ be a hypergraph. If  $\widetilde{w}_{\mathcal{H}}:\mathcal{H}\rightarrow \mathbb{Q}(\sqrt{q})$ or $\mathbb{Z}[\![\sqrt{q}]\!]$ is a function satisfying the weighting equations, we have $\widetilde{w}_\mathcal{H}=w_\mathcal{H}$ and $\# \mathcal{H}= \sum\limits_{\sigma\in \mathcal{H}}\widetilde{w}_\mathcal{H}(\sigma)$.
\end{lemma}
\begin{proof}
We have known the matrix $Z_{\mathcal{H}}(q)$ is invertible over $\mathbb{Q}(\sqrt{q})$ or $\mathbb{Z}[\![\sqrt{q}]\!]$. Therefore, by the weighting equations, we can obtain $w_{\mathcal{H}}:\mathcal{H}\rightarrow \mathbb{Q}(\sqrt{q})$ or $\mathbb{Z}[\![\sqrt{q}]\!]$ is unique. So $\widetilde{w}_\mathcal{H}=w_\mathcal{H}$.

\end{proof}
\begin{proposition}\label{prop2.5}
Let $\mathcal{H}$ be a hypergraph with the hyperedges $\sigma_{0},\dots,\sigma_{n}$. Then we have
\begin{equation*}
  \# \mathcal{H}=\sum_{n=0}^{\infty}(-1)^{n}\sum_{\sigma_{i_p}\neq \sigma_{i_{p+1}},p=0,\dots,n-1} q^{d(\sigma_{i_0},\sigma_{i_1})+\cdots+ d(\sigma_{i_{n-1}},\sigma_{i_n})},
\end{equation*}
where $\sigma_{i_0},\dots,\sigma_{i_n}$ are the arbitrary combination of the hyperedges of $\mathcal{H}$.
\end{proposition}
\begin{proof}
For each $\sigma\in \mathcal{H}$, we define $\widetilde{w}_\mathcal{H}(\sigma)\in Z[\![\sqrt{q}]\!]$ by
$$
\widetilde{w}_\mathcal{H}(\sigma)=\sum_{n=0}^{\infty}(-1)^{n}\sum_{\sigma=\sigma_{i_0}, \sigma_{i_p}\neq \sigma_{i_{p+1}},p=0,\dots,n-1} q^{d(\sigma_{i_0},\sigma_{i_1})+\cdots+ d(\sigma_{i_{n-1}},\sigma_{i_n})}.
$$
By the lemma \ref{lemma:weight}, we have known that if $\widetilde{w}_\mathcal{H}(\sigma)$ satisfies weighting equation, we have $\# \mathcal{H}= \sum\limits_{\sigma\in \mathcal{H}}\widetilde{w}_\mathcal{H}(\sigma)$.
Therefore,
\begin{equation*}
  \begin{split}
      \sum_{\tau\in \mathcal{H}}q^{d(\sigma,\tau)}\widetilde{w}_\mathcal{H}(\tau)=& \widetilde{w}_\mathcal{H}(\sigma)+ \sum_{\tau:\tau\neq\sigma}q^{d(\sigma,\tau)}\widetilde{w}_\mathcal{H}(\tau) \\
    =&\widetilde{w}_\mathcal{H}(\sigma)+\sum_{n=0}^{\infty}(-1)^{n}\sum_{\sigma\neq\tau_{i_0},\tau_{i_p}\neq \tau_{i_{p+1}},p=0,\dots,n-1} q^{d(\sigma,\tau_{i_0})+d(\tau_{i_0},\tau_{i_1})+\cdots+ d(\tau_{i_{n-1}},\tau_{i_n})}\\
     =&1+\sum_{n=1}^{\infty}(-1)^{n}\sum_{\sigma_{i_p}\neq \sigma_{i_{p+1}},p=1,\dots,n-1} q^{d(\sigma_{i_0},\sigma_{i_1})+d(\sigma_{i_1},\sigma_{i_2})+\cdots+ d(\sigma_{i_{n-1}},\sigma_{i_n})}+\\
      &\sum_{m=0}^{\infty}(-1)^{m}\sum_{\sigma\neq\tau_{i_0},\tau_{i_p}\neq \tau_{i_{p+1}},p=0,\dots,m-1} q^{d(\sigma,\tau_{i_0})+d(\tau_{i_0},\tau_{i_1})+\cdots+ d(\tau_{i_{m-1}},\tau_{i_m})}\\
  \end{split}
\end{equation*}
Let $n-1=m$, we can convert the formula $$\sum_{n=1}^{\infty}(-1)^{n}\sum_{\sigma_{i_p}\neq \sigma_{i_{p+1}},p=1,\dots,n-1} q^{d(\sigma_{i_0},\sigma_{i_1})+d(\sigma_{i_1},\sigma_{i_2})+\cdots+ d(\sigma_{i_{n-1}},\sigma_{i_n})}.$$
By calculation,  $\sum\limits_{\tau\in \mathcal{H}}q^{d(\sigma,\tau)}\widetilde{w}_\mathcal{H}(\tau)=1$.
\end{proof}

\begin{example}\label{em:2.4}
(\textbf{The magnitude of the $2$-simplex})\label{exam.2.1}
Given a simplicial complex $\Delta[2]$=$\{\{0\},$ $\{1\},\{2\}, \{0,1\}, \{0,2\}, \{1,2\}, \{0,1,2\}\}$ as an example, we will now examine its magnitude. The following table, which lists the distances between simplices, serves as a useful tool for calculating the magnitude $\#\mathcal{H}$.
\begin{table}[H]
  \centering
  \caption{The distance of simplices}\label{table:distance1}
  \begin{tabular}{c|ccccccc}
    \hline
    distance & \{0\}& \{1\}&\{2\}&\{0, 1\}&\{0, 2\}&\{1, 2\}&\{0, 1, 2\}\\
        \hline
    \{0\}& 0 & 1 & 1 & $\frac{1}{2}$ & $\frac{1}{2}$ & 1 & $\frac{1}{2}$ \\

    \{1\}& 1 & 0 & 1 & $\frac{1}{2}$ & 1 & $\frac{1}{2}$ & $\frac{1}{2}$ \\

    \{2\}& 1 & 1 & 0 & 1 & $\frac{1}{2}$ & $\frac{1}{2}$ & $\frac{1}{2}$ \\

    \{0, 1\}& $\frac{1}{2}$ & $\frac{1}{2}$ & 1 & 0 & 1 & 1 & $\frac{1}{2}$ \\

    \{0, 2\}& $\frac{1}{2}$ & 1 & $\frac{1}{2}$ & 1 & 0 & 1 & $\frac{1}{2}$ \\

    \{1, 2\}& 1 & $\frac{1}{2}$ & $\frac{1}{2}$ & 1 & 1 & 0 & $\frac{1}{2}$ \\

    \{0, 1, 2\} & $\frac{1}{2}$ & $\frac{1}{2}$ & $\frac{1}{2}$ & $\frac{1}{2}$ & $\frac{1}{2}$ & $\frac{1}{2}$ & 0\\
    \hline
  \end{tabular}

\end{table}

By Proposition \ref{prop2.5}, we have
\begin{itemize}
  \item[(i)] When $n=0$, the first term of $\#\mathcal{H}$ is $\#\mathcal{H}_1= (-1)^0 \times \sum\limits_{i=0}^6 q^{d(\sigma_i, \sigma_i)}= (-1)^0 \times 7$.
  \item[(ii)] When $n=1$, then we have the sum of the first two terms of $\#\mathcal{H}$ is
  \begin{equation*}
  \begin{split}
  \#\mathcal{H}_2&= (-1)^0 \times \sum\limits_{i=0}^6 q^{d(\sigma_i, \sigma_i)} + (-1)^1 \times \sum\limits_{\substack{\sigma_i\neq \sigma_j}} q^{d(\sigma_i, \sigma_j)}\\
               &= (-1)^0 \times 7 + (-1)^1\times(6 (q+q+q+q^{\frac{1}{2}}+q^{\frac{1}{2}}+q^{\frac{1}{2}})+ 6q^{\frac{1}{2}})\\
               &= (-1)^0 \times 7 + (-1)^1\times(6 (3q+3q^{\frac{1}{2}})+6q^{\frac{1}{2}})
  \end{split}
  \end{equation*}
\item [(iii)] When $n=2$, we can get
\small{  \begin{equation*}
  \begin{split}
  \#\mathcal{H}&= (-1)^0 \times \sum\limits_{i=0}^6 q^{d(\sigma_i, \sigma_i)} + (-1)^1 \times \sum\limits_{\substack{\sigma_i\neq \sigma_j}} q^{d(\sigma_i, \sigma_{j})}+ (-1)^2 \times \sum\limits_{\substack{\sigma_i\neq \sigma_{j},\sigma_j\neq \sigma_{k}}} q^{d(\sigma_i, \sigma_j)+d(\sigma_j, \sigma_k)}\\
              &= (-1)^0 \times 7 + (-1)^1\times(6 (3q+3q^{\frac{1}{2}})+6q^{\frac{1}{2}})+ (-1)^2\times(6 (3q+3q^{\frac{1}{2}})^2+(6q^{\frac{1}{2}})^2)
  \end{split}
  \end{equation*}}
\end{itemize}
Continuing the calculation process described above, it is worth noting that since $n$ can be infinite, we will provide a partial expression for $\Delta[2]$,
\begin{equation*}
  \begin{split}
  \#\mathcal{H}=& (-1)^{0}\times7 + (-1)^1\times(6 (3q+3q^{\frac{1}{2}})+6q^{\frac{1}{2}})+ (-1)^2\times(6 (3q+3q^{\frac{1}{2}})^2+(6q^{\frac{1}{2}})^2)+\cdots\\
                & (-1)^{6}\times(6 (3q+3q^{\frac{1}{2}})^6+(6q^{\frac{1}{2}})^6)+\cdots\\
               =&\sum^{\infty}_{n=0}(-1)^{n}(6 (3q+3q^{\frac{1}{2}})^n+(6q^{\frac{1}{2}})^n)+\cdots\\
               =& 7 - 24 q^{\frac{1}{2}} + 72 q - 270 q^{\frac{3}{2}} + 1350 q^2 + 1458 q^{\frac{5}{2}} +
                      2754 q^3 + 1944 q^{\frac{7}{2}} + 486 q^4+ \cdots.
  \end{split}
\end{equation*}

\end{example}

\begin{example}\label{em:2.5}
(\textbf{The magnitude of a hypergraph})
Consider the hypergraph $\mathcal{H}=\{\{0\},\{1\},$ $\{2\},\{0,1\},\{1,2\},\{0,1,2\}\}$ ( Fig.\ref{fig:hypergraph}(a)). By utilizing Proposition \ref{prop2.5}, we can calculate the magnitude of the hypergraph. Employing the previously defined notion of distance between hyperedges, we derive the following table that depicts the distances between simplices.
\begin{table}[H]
  \centering
  \caption{The distance of simplices}\label{table:distance2}
  \begin{tabular}{c|ccccccc}
    \hline
    distance & \{0\}& \{1\}&\{2\}&\{0, 1\}&\{1, 2\}&\{0, 1, 2\}\\
        \hline
    \{0\}& 0 & 1 & 1 & $\frac{1}{2}$ &  1 & $\frac{1}{2}$ \\

    \{1\}& 1 & 0 & 1 & $\frac{1}{2}$ & $\frac{1}{2}$ & $\frac{1}{2}$ \\

    \{2\}& 1 & 1 & 0 & 1 & $\frac{1}{2}$ & $\frac{1}{2}$ \\

    \{0, 1\}& $\frac{1}{2}$ & $\frac{1}{2}$ & 1 & 0 & 1 & $\frac{1}{2}$ \\

    \{1, 2\}& 1 & $\frac{1}{2}$ & $\frac{1}{2}$ & 1 & 0 & $\frac{1}{2}$ \\

    \{0, 1, 2\} & $\frac{1}{2}$ & $\frac{1}{2}$ & $\frac{1}{2}$ & $\frac{1}{2}$ & $\frac{1}{2}$ & 0\\
    \hline
  \end{tabular}

\end{table}
By the same calculation method as in Example \ref{exam.2.1}, we can obtain
\begin{equation*}
  \begin{split}
  \#\mathcal{H}=& (-1)^{0}\times6 + (-1)^{1}\times(2(3q+2q^{\frac{1}{2}})+3(2q+3q^{\frac{1}{2}})+5q^{\frac{1}{2}})\\
                & (-1)^{2}\times(2(3q+2q^{\frac{1}{2}})^2+3(2q+3q^{\frac{1}{2}})^2+(5q^{\frac{1}{2}})^2)+\cdots\\
                & (-1)^{5}\times(2(3q+2q^{\frac{1}{2}})^5+3(2q+3q^{\frac{1}{2}})^5+(5q^{\frac{1}{2}})^5)+\cdots\\
               =&\sum^{\infty}_{n=0}(-1)^{n}(2(3q+2q^{\frac{1}{2}})^n+3(2q+3q^{\frac{1}{2}})^n+(5q^{\frac{1}{2}})^n)\\
               =& 6 - 18q^{\frac{1}{2}}  + 48 q - 162 q^{\frac{3}{2}} + 696 q^2 + 624 q^{\frac{5}{2}} +1002 q^3 + 720 q^{\frac{7}{2}} + 210 q^4+ \cdots
  \end{split}
\end{equation*}

\begin{figure}[H]
\begin{center}
  \centering
  \subfigure[]{
  \begin{minipage}[c]{0.47\textwidth}
  \centering
  \includegraphics[scale=0.4]{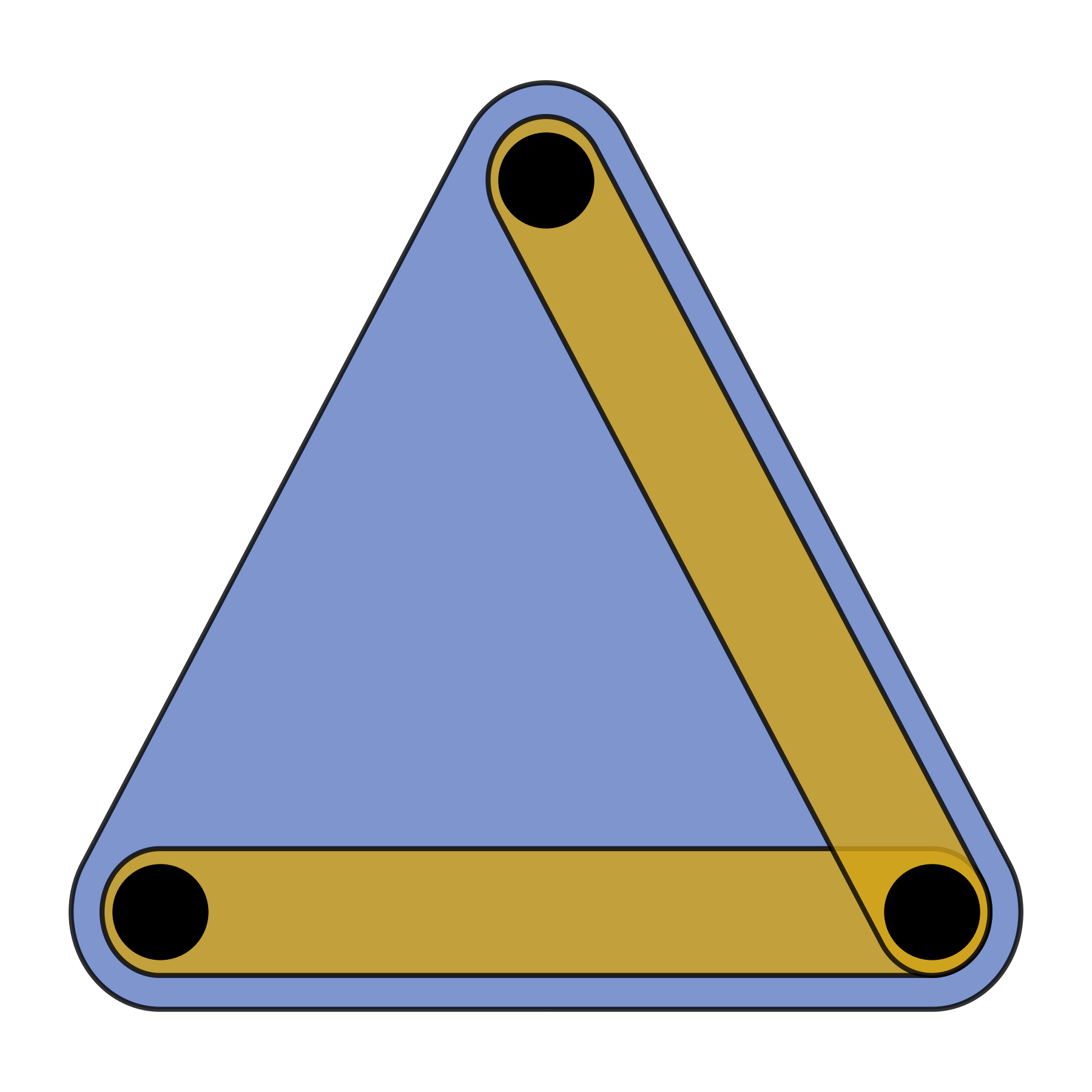}
 \end{minipage}
   }
 \subfigure[]{
 \begin{minipage}[c]{0.47\textwidth}
 \centering
 \includegraphics[scale=0.4]{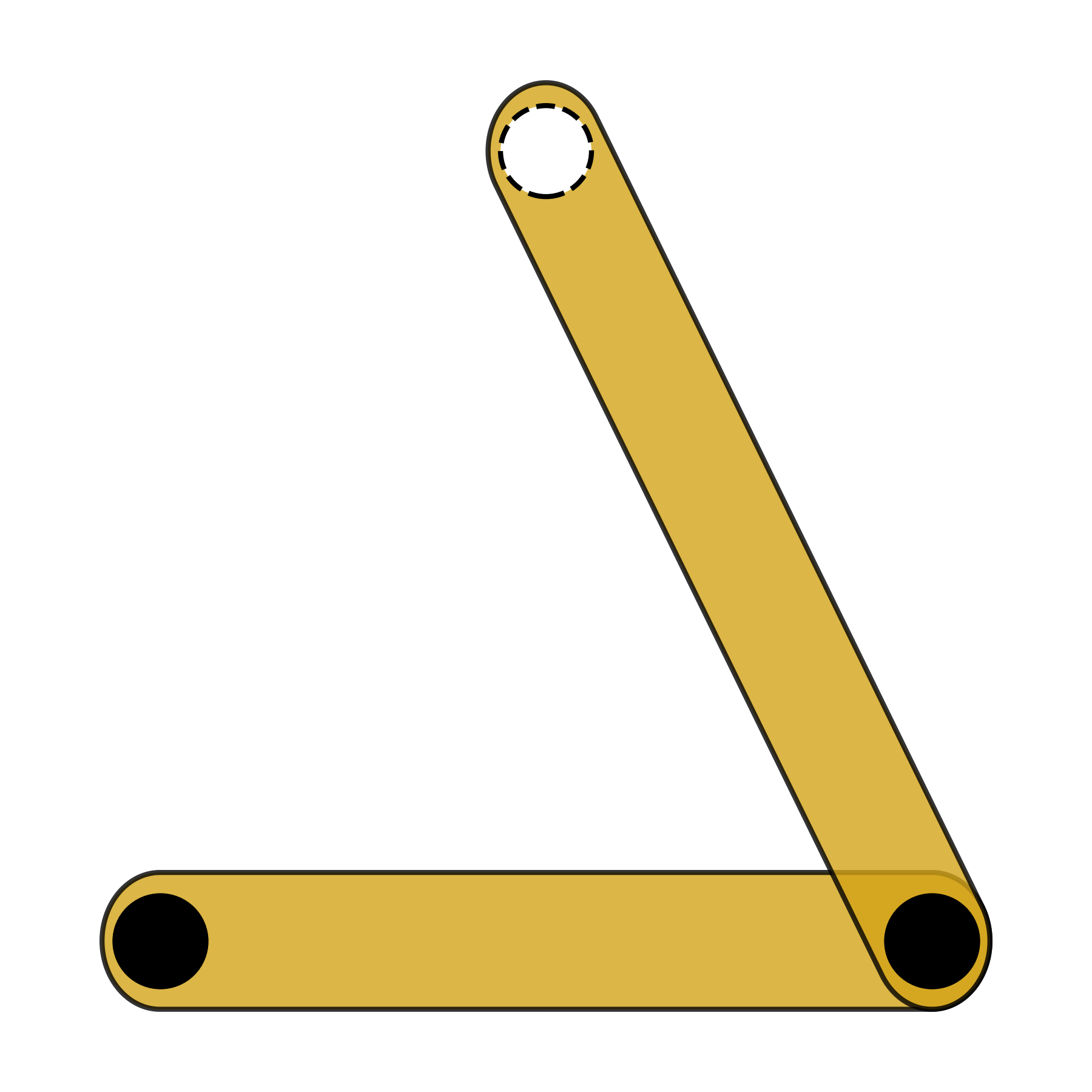}
 \end{minipage}
 }
 \caption{(a) The hypergraph $\mathcal{H}=\{\{0\},\{1\},$ $\{2\},\{0,1\},\{1,2\},\{0,1,2\}\}$; (b) The hypergraph $\mathcal{H}=\{\{0\},\{1\},\{0,1\},\{1,2\}\}$.}
 \label{fig:hypergraph}
\end{center}
\end{figure}

\end{example}

\subsection{The magnitude homology of a hypergraph}\label{section:magnitude}
Let $\mathcal{H}$ be a hypergraph. The \emph{length of the tuple}  $(\sigma_{0},\sigma_{1},\dots,\sigma_{k})$ of hyperedges of $\mathcal{H}$ is defined by
\begin{equation*}
  \mathcal{L}(\sigma_{0},\sigma_{1},\dots,\sigma_{k})=d(\sigma_{0},\sigma_{1})+\cdots+d(\sigma_{k-1},\sigma_{k}).
\end{equation*}
Note that the intercrossing distance $d$ is an extended metric on $\mathcal{H}$. By the triangle inequality, one has
\begin{equation}
\mathcal{L}(\sigma_0, \dots, \hat{\sigma_i}, \dots, \sigma_k)\leq \mathcal{L}(\sigma_0, \dots, \sigma_k), \label{(1)}
\end{equation}
where $\hat{\sigma_{i}}$ means the omission of the term $\sigma_{i}$.
\begin{definition}\label{def:2.5}
Let $A$ be an abelian group. The \emph{magnitude chain complex $\mathbf{MC}_{\ast,\ast}(\mathcal{H};A)$} of a hypergraph $\mathcal{H}$ with coefficient $A$ is a chain complex $\bigoplus\limits_{l\geq 0}\mathbf{MC}_{\ast,l}(\mathcal{H};A)$ defined as follows. For each integer $k\geq 0$, $\mathbf{MC}_{k,l}(\mathcal{H};A)$ is a free abelian group over $A$ generated by the tuples $(\sigma_{0},\dots,\sigma_{k})$ of adjacent distinct hyperedges satisfying $\mathcal{L}(\sigma_{0},\dots,\sigma_{k})=l$. The differential $\partial_{l}:\mathbf{MC}_{k,l}(\mathcal{H};A)\to \mathbf{MC}_{k-1,l}(\mathcal{H};A)$ is given by $\partial_{l}=\sum\limits_{i=0}^{k}(-1)^i\partial_{i,l}$ with
\begin{equation*}
  \partial_{i,l}(\sigma_{0},\dots,\sigma_{k})=\left\{
                                              \begin{array}{ll}
                                                (\sigma_{0},\dots,\hat{\sigma_{i}},\dots,\sigma_{k}) & \hbox{$\mathcal{L}(\sigma_{0},\dots,\hat{\sigma_{i}},\dots,\sigma_{k})=l$}, \\
                                                0 & \hbox{otherwise.}
                                              \end{array}
                                            \right.
\end{equation*}
\end{definition}

\begin{lemma}\label{lem:different}
For any hypergraph $\mathcal{H}$, when $k\geq 2$, $l\geq 0$, we have $\partial_l\circ \partial_l=0$, that is, the composition
$$
\mathbf{MC}_{k,l}(\mathcal{H})\stackrel{\partial_l}{\longrightarrow}\mathbf{MC}_{k-1,l}(\mathcal{H})\stackrel{\partial_l}{\longrightarrow}\mathbf{MC}_{k-2,l}(\mathcal{H})
$$
is zero.
\end{lemma}

\begin{proof}
Our aim is to prove $\sum\limits_{i=0}^{k-1}(-1)^i\partial_{i,l}\circ \sum\limits_{j=0}^{k}(-1)^j\partial_{j,l}=0,$ where $0\leq i< j \leq k$. It suffices to show
$$
\partial_{i,l}\circ \partial_{j,l}(\sigma_0,\dots,\sigma_k)=\partial_{j-1,l}\circ \partial_{i, l} (\sigma_0,\dots,\sigma_k)
$$
for any generator $(\sigma_0,\dots,\sigma_k)$ of $\mathbf{MC}_{k,l}(\mathcal{H})$. By definition, we have
$$\mathcal{L}(\sigma_0,\dots, \hat{\sigma_i},\dots, \hat{\sigma_j}, \dots, \sigma_k)\leq \mathcal{L}(\sigma_0,\dots, \hat{\sigma_i}, \dots, \sigma_k)\leq \mathcal{L}(\sigma_0,\dots,\sigma_k)=l.$$
If $\mathcal{L}(\sigma_0,\dots, \hat{\sigma_i},\dots, \hat{\sigma_j}, \dots, \sigma_k)< \mathcal{L}(\sigma_0,\dots,\sigma_k)=l,$ both sides are equal to zero. The desired equation follows. If
$$
\mathcal{L}(\sigma_0,\dots, \hat{\sigma_i},\dots, \hat{\sigma_j}, \dots, \sigma_k)= \mathcal{L}(\sigma_0,\dots, \hat{\sigma_i}, \dots, \sigma_k)= \mathcal{L}(\sigma_0,\dots,\sigma_k)=l,
$$
both sides are of the form $(\sigma_0,\dots, \hat{\sigma_i},\dots, \hat{\sigma_j}, \dots, \sigma_k)$. It follows that $\partial_{i,l}\circ \partial_{j,l}(\sigma_0,\dots,\sigma_k)=\partial_{j-1,l}\circ \partial_{i, l}(\sigma_0,\dots,\sigma_k)$.
Therefore, we have $\sum\limits_{i=0}^{k-1}(-1)^i\partial_{i,l}\circ \sum\limits_{j=0}^{k}(-1)^j\partial_{j,l}=0.$ This completes the proof.
\end{proof}

\begin{definition}
The \emph{magnitude homology $\mathbf{MH}_{k,l}(\mathcal{H};A)$} of the hypergraph $\mathcal{H}$ with coefficient $A$ is defined by the homology of the magnitude chain complex
\begin{equation*}
  \mathbf{MH}_{k,l}(\mathcal{H};A)=H_{k,l}(\mathbf{MC}_{\ast,\ast}(\mathcal{H};A)),\quad k,l\geq 0.
\end{equation*}
\end{definition}
From now on, the ground ring is assumed to be $\mathbb{Z}$, and we denote $\mathbf{MH}_{n,l}(\mathcal{H})=\mathbf{MH}_{n,l}(\mathcal{H};\mathbb{Z})$ for convenience.

\begin{example}\label{em:2.6}
(\textbf{The magnitude homology of a hypergraph})
Given a hypergraph $\mathcal{H}=\{\{0\},\{1\},\{0,1\},\{1,2\}\}$, we will show its magnitude homology below(see  Fig.\ref{fig:hypergraph}(b)). According to the definition of magnitude homology, there are two significant values: integers $k$ and length $l$. We also know that $k$ is infinite, and the calculation becomes complicated due to the definition of magnitude homology. Next, we will discuss the different conditions and their implications. Here, we only consider values of $k$ up to 2.
\begin{itemize}
  \item [(i)] When $k=0$, $l=0$, we have the generators $(\{0\})$, $(\{1\})$, $(\{0, 1\})$, $(\{1, 2\})$.
  \item [(ii)] When $k=1$, $l$ has two following cases.
  \begin{itemize}
   \item[(a)]$l=\frac{1}{2}$, we have generators $(\{0\}, \{0, 1\})$, $(\{1\}, \{0,1\})$, $(\{1\}, \{1, 2\})$, $(\{0, 1\}, \{0\})$, $(\{0, 1\},$ $ \{1\})$, $( \{1, 2\},\{1\})$.
  \item[(b)]$l=1$, we get $(\{0\}, \{1\})$, $(\{0\}, \{1, 2\})$, $(\{1\}, \{0\})$, $(\{1, 2\}, \{0\})$, $(\{0, 1\}, \{1, 2\})$, $(\{1, 2\}, \{0, 1\})$.
  \end{itemize}
  \item [(iii)] When $k=2$, there is three options for $l$.
  \begin{itemize}
   \item[(a)]$l=1$, generators are  $(\{0\}, \{0, 1\}, \{0\})$, $(\{0\}, \{0, 1\}, \{1\})$, $(\{1\}, \{0,1\}, \{0\})$, $(\{1\}, \{0,1\},$ $\{1\})$, $(\{1\}, \{1, 2\}, \{1\})$, $(\{0, 1\}, \{0\}, \{0, 1\})$, $(\{0, 1\}, \{1\}, \{0, 1\})$, $(\{0, 1\}, \{1\}, \{1, 2\})$, $( \{1, 2\},\{1\}, \{0, 1\})$, $( \{1, 2\},\{1\}, \{1, 2\})$.
  \item[(b)]$l=\frac{3}{2}$, we obtain $(\{0\}, \{0, 1\}, \{1, 2\})$, $(\{1\}, \{0,1\}, \{1, 2\})$, $(\{1\}, \{1, 2\}, \{0\})$, $(\{1\},$ $\{1, 2\}, \{0, 1\})$, $(\{0, 1\}, \{0\}, \{1\})$, $(\{0, 1\}, \{0\}, \{1, 2\})$, $(\{0, 1\}, \{1\}, \{0\})$, $(\{1, 2\}, \{1\},$ $\{0\})$, $(\{0\}, \{1\}, \{0, 1\})$, $(\{0\}, \{1\}, \{1, 2\})$, $(\{0\}, \{1, 2\}, \{1\})$, $(\{1\}, \{0\}, \{0, 1\})$, $(\{1, 2\},$ $\{0\}, \{0, 1\})$, $(\{0, 1\}, \{1, 2\}, \{1\})$, $(\{1, 2\}, \{0, 1\}, \{0\})$, $(\{1, 2\}, \{0, 1\}, \{1\})$.
  \item[(c)]$l=2$, we get $(\{0\}, \{1\}, \{0\})$, $(\{0\}, \{1, 2\}, \{0\})$, $(\{0\}, \{1, 2\}, \{0, 1\})$, $(\{1\}, \{0\}, \{1\})$, $(\{1\}, \{0\}, \{1, 2\})$, $(\{1, 2\}, \{0\}, \{1\})$, $(\{1, 2\}, \{0\}, \{1, 2\})$, $(\{0, 1\}, \{1, 2\}, \{0, 1\})$, $(\{0, 1\},$ $\{1, 2\}, \{0\})$, $(\{1, 2\}, \{0, 1\}, \{1, 2\})$ .
  \end{itemize}
\end{itemize}

Based on the above arrays, we consider the magnitude chain complex, where $l=0, \frac{1}{2}, 1$ is given as an example. The chain complexes are given by
$$
\cdots\longrightarrow \mathbf{MC}_{2,0}(\mathcal{H})\stackrel{\partial_0}\longrightarrow \mathbf{MC}_{1,0}(\mathcal{H})\stackrel{\partial_0}\longrightarrow \mathbf{MC}_{0,0}(\mathcal{H})\longrightarrow 0,
$$
$$
\cdots\longrightarrow \mathbf{MC}_{2,\frac{1}{2}}(\mathcal{H})\stackrel{\partial_{\frac{1}{2}}}\longrightarrow \mathbf{MC}_{1,\frac{1}{2}}(\mathcal{H})\stackrel{\partial_{\frac{1}{2}}}\longrightarrow \mathbf{MC}_{0,\frac{1}{2}}(\mathcal{H})\longrightarrow 0,
$$
and
$$
\cdots\longrightarrow \mathbf{MC}_{2,1}(\mathcal{H})\stackrel{\partial_1}\longrightarrow \mathbf{MC}_{1,1}(\mathcal{H})\stackrel{\partial_1}\longrightarrow \mathbf{MC}_{0,1}(\mathcal{H})\longrightarrow 0.
$$
From the first magnitude chain complex, we can calculate $\mathbf{MH}_{0,0}(\mathcal{H})=\mathbb{Z}\oplus \mathbb{Z}\oplus \mathbb{Z} \oplus \mathbb{Z}$. Because the term $\mathbf{MC}_{1,0}(\mathcal{H})$ is zero, the kernel of $\mathbf{MC}_{0,0}(\mathcal{H})$ is generated by $\mathbb{Z}\langle(\{0\}), (\{1\}), (\{0, 1\}),$ $(\{1, 2\})\rangle$. By performing the same calculation, we also obtain $\mathbf{MH}_{k,0}(\mathcal{H})=0$, $k\neq 0$. According to the second magnitude chain complex, repeating the above calculation, we get $\mathbf{MH}_{1,\frac{1}{2}}(\mathcal{H})= \mathbb{Z}\oplus \mathbb{Z} \oplus \mathbb{Z} \oplus \mathbb{Z} \oplus \mathbb{Z} \oplus \mathbb{Z}$, zero in other cases. When calculating the case for $l=1$, we can only get $\mathbf{MH}_{1,1}(\mathcal{H})=\mathbb{Z}\oplus \mathbb{Z}$ based on the array currently listed, and if we continue further, the computational workload will progressively increase. Readers who are interested can continue with the calculations.
\end{example}
Here, we have only considered simple hypergraphs. Because as we know from the paper \cite{hepworth2017categorifying}, even though the magnitude homology of the graphs is already large, the authors used a Sage (Computer Software) for the calculation. We just provide here the method of calculating the magnitude homology of hypergraphs.

\begin{theorem}\label{Eluer}
Let $\mathcal{H}$ be a hypergraph. Then
\begin{equation*}
  \sum\limits_{k,l\geq 0}(-1)^{k}\mathrm{rank}(\mathbf{MH}_{k,l}(\mathcal{H}))\cdot q^{l}=\#\mathcal{H}.
\end{equation*}
Here, $n\in \mathbb{Z},l\in \frac{1}{2}\mathbb{Z}$.
\end{theorem}
\begin{proof}
The proof is parallel to the proof of \cite[Theorem 8]{hepworth2017categorifying}.
Let $\chi$ denote the ordinary Euler characteristic of chain complexes.
Noting that $\chi(M_{\ast})=\chi(H_{\ast}(M_{\ast}))$ for a chain complex $M_{\ast}$, we have
\begin{equation*}
  \begin{split}
    \sum\limits_{k,l\geq 0}(-1)^{k}\mathrm{rank}(\mathbf{MH}_{k,l}(\mathcal{H}))\cdot q^{l} =& \sum\limits_{l\geq 0}(-1)^{k}\chi(\mathbf{MH}_{\ast,l}(\mathcal{H}))\cdot q^{l} \\
      =& \sum\limits_{l\geq 0}(-1)^{k}\chi(\mathbf{MC}_{\ast,l}(\mathcal{H}))\cdot q^{l}\\
    =&\sum\limits_{k,l\geq 0}(-1)^{k}\mathrm{rank}(\mathbf{MC}_{k,l}(\mathcal{H}))\cdot q^{l}\\
     =&\sum\limits_{k\geq 0}(-1)^{k}\sum_{\sigma_{i}\neq \sigma_{i+1},i=1,\dots,k-1} q^{d(\sigma_0,\sigma_1)+\cdots+ d(\sigma_{k-1},\sigma_{k})}.
  \end{split}
\end{equation*}
By Proposition \ref{prop2.5}, we have the desired result.
\end{proof}
From the above discussion of magnitude homology, we can derive the following properties.
\begin{proposition}\label{prop:2.9}
Let $\mathcal{H}$ be a hypergraph. Then $\mathbf{MH}_{0,0}(\mathcal{H})$ is the free abelian group generated by the hyperedges of $\mathcal{H}$ and $\mathbf{MH}_{1,\frac{1}{2}}(\mathcal{H})$ is the free abelian group generated by the pairs $(\sigma, \tau)$ of hyperedges with $\sigma\subsetneq\tau$ or $\tau\subsetneq \sigma$.
\end{proposition}
\begin{proof}
When calculating $\mathbf{MH}_{0,0}(\mathcal{H})$ and $\mathbf{MH}_{1,\frac{1}{2}}(\mathcal{H})$, the chain complex $\mathbf{MC}_{0,0}(\mathcal{H})$ and $\mathbf{MC}_{1,\frac{1}{2}}(\mathcal{H})$ will be considered. Also because all differentials about the above chain complex are zero. The proof is completed.
\end{proof}

\section{The functorial properties}\label{section:functorial}
We have known that the magnitude homology of a graph is an object of the category of bigraded abelian groups. And the functorial properties of the magnitude were also introduced in \cite{hepworth2017categorifying}. Next, we will have a similar result on the hypergraph.

To construct the category of a hypergraph, we choose the following morphism whose objects are hypergraphs.
Let $\mathcal{G}$ and $\mathcal{H}$ be two hypergraphs, a map of hypergraphs $f:\mathcal{G}\rightarrow \mathcal{H}$ is a map of hyperedge sets $f:\sigma({\mathcal{G}})\rightarrow\sigma({\mathcal{H}})$ satisfying that $f(\sigma_i)\subseteq f(\sigma_j)$ whenever $\sigma_i\subseteq\sigma_j.$

Considering the distance of hyperedges, we have $d_\mathcal{H}(f(\sigma_i),f(\sigma_j))\leq d_\mathcal{G}(\sigma_i,\sigma_j)$ for all $\sigma_i,\sigma_j\in\sigma({\mathcal{G}}).$ If $f:\mathcal{G}\rightarrow \mathcal{H}$ is a map of hypergraphs, for any tuple $(\sigma_0,\dots,\sigma_k)$ of hyperedges of $\mathcal{G}$, we still have the inequality
\begin{equation}
\mathcal{L}(f(\sigma_0),\dots,f(\sigma_k))\leq \mathcal{L}(\sigma_0,\dots,\sigma_k) \label{(2)}.
\end{equation}

\begin{definition}{\label{Induced}}
If $f:\mathcal{G}\rightarrow \mathcal{H}$ is a map of hypergraphs, then we define the \emph{induced chain map} $f_{\#}:\mathbf{MC}_{*,*}(\mathcal{G})\rightarrow \mathbf{MC}_{*,*}(\mathcal{H}), $
which is defined on generators by
\[
f_{\#}(\sigma_0,\dots,\sigma_k)=\begin{cases}
(f(\sigma_0),\dots,f(\sigma_k)) & {\rm if} \ \mathcal{L}(f(\sigma_0),\dots,f(\sigma_k))=\mathcal{L}(\sigma_0,\dots,\sigma_k),\\
0& {\rm otherwise}.
\end{cases}
\]
\end{definition}

With the notion of Definition \ref{Induced}, we have the following proposition.
\begin{proposition}
The map $f_{\#}$ is a chain map.
\end{proposition}
\begin{proof}
To show the map $f_{\#}$ is a chain map, it suffices to prove the diagram
$$
\xymatrix{
  \mathbf{MC}_{k,*}(\mathcal{G}) \ar[d]_{f_{\#}} \ar[r]^{\partial_l} & \mathbf{MC}_{k-1,*}(\mathcal{G}) \ar[d]^{f_{\#}} \\
  \mathbf{MC}_{k,*}(\mathcal{H}) \ar[r]^{\partial_l} & \mathbf{MC}_{k-1,*}(\mathcal{H})  }
$$
commutes. By definition, we have
$$
\mathcal{L}(f(\sigma_0),\dots ,\widehat{f(\sigma_i)} ,\dots,f(\sigma_k))\leq \mathcal{L}(\sigma_0,\dots,\widehat{\sigma_i},\dots,\sigma_k)\leq \mathcal{L}(\sigma_0,\dots,\sigma_k)=l.
$$
When $\mathcal{L}(f(\sigma_0),\dots ,\widehat{f(\sigma_i)} ,\dots,f(\sigma_k))< \mathcal{L}(\sigma_0,\dots,\sigma_k)=l$, both sides of the equation $f_{\#}\circ \partial_{i,l}=\partial_{i,l}\circ f_{\#}$ are zero. If
$$
\mathcal{L}(f(\sigma_0),\dots ,\widehat{f(\sigma_i)} ,\dots,f(\sigma_k))= \mathcal{L}(\sigma_0,\dots,\widehat{\sigma_i},\dots,\sigma_k)= \mathcal{L}(\sigma_0,\dots,\sigma_k)=l,
$$
we have $f_{\#}\circ \partial_{i,l}(\sigma_0,\dots,\sigma_k)=(f(\sigma_0),\dots ,\widehat{f(\sigma_i)} ,\dots,f(\sigma_k))=\partial_{i,l}\circ f_{\#}(\sigma_0,\dots,\sigma_k).$ The desired result is obtained.
\end{proof}
\begin{definition}\label{induced homology}
If $f:\mathcal{G}\rightarrow \mathcal{H}$ is a map of hypergraphs, then we define the homomorphisms
$f_{*}:\mathbf{MH}_{*,*}(\mathcal{G})\rightarrow \mathbf{MH}_{*,*}(\mathcal{H})$ as the induced homomorphisms of $f_{\#}$.
\end{definition}

\begin{proposition}
Let $\mathcal{G}$ be a hypergraph, we have $\mathcal{G}\mapsto \mathbf{MH}_{*,*}(\mathcal{G})$, $f\mapsto f_*$ is a functor, which is from the category of hypergraphs to the category of bigraded abelian groups.
\end{proposition}

It is evident that the identity map of a hypergraph induces the identity map in homology. For any maps of hypergraphs $f:\mathcal{G}\rightarrow \mathcal{H}$ and $g:\mathcal{H}\rightarrow \mathcal{K}$, by the lemma $\ref{lem:different}$ and the inequality ($\ref{(1)}$) and ($\ref{(2)}$), we can prove $g_*\circ f_*=(g\circ f)_*$. Here, the details of the proof are left to the reader.

Combining Proposition \ref{prop:2.9} and the induced map in homology, we have the following results, whose proof is immediate to obtain from the definitions.

\begin{proposition}
Let $f:\mathcal{G}\rightarrow \mathcal{H}$ be a map of hypergraphs.
\begin{itemize}
\item[$(i)$]If $f_*:\mathbf{MH}_{0,0}(\mathcal{G})\rightarrow \mathbf{MH}_{0,0}(\mathcal{H})$ , it maps a hyperedge $\sigma$ to $f(\sigma)$ in $\mathcal{H}$.
\item[$(ii)$]If $f_*:\mathbf{MH}_{1,\frac{1}{2}}(\mathcal{G})\rightarrow \mathbf{MH}_{1,\frac{1}{2}}(\mathcal{H})$, it maps the pairs $(\sigma,\tau)$ with $\sigma\subsetneq\tau$ or $\tau\subsetneq \sigma$ to $(f(\sigma),f(\tau))$ satisfying $f(\sigma)\subsetneq f(\tau)$ or $f(\tau)\subsetneq f(\sigma)$, otherwise maps to 0.
\end{itemize}
\end{proposition}

\begin{corollary}
Let $f:\mathcal{G}\rightarrow \mathcal{H}$ be a map of hypergraphs. If $f_*:\mathbf{MH}_{*,*}(\mathcal{G})\rightarrow \mathbf{MH}_{*,*}(\mathcal{H})$ is an isomorphism, then $f$ is an isomorphism of hypergraphs.
\end{corollary}

Next, we'll prove the additivity of magnitude homology concerning disjoint unions.
\begin{proposition}\label{disjoint}
Given hypergraphs $\mathcal{G}$ and $\mathcal{H}$, the inclusion maps are denoted by $i:\mathcal{G}\rightarrow \mathcal{G}\sqcup \mathcal{H}$ and $j:\mathcal{H}\rightarrow \mathcal{G}\sqcup \mathcal{H}$ respectively, then the induced map
$$
i_*\oplus j_*:\mathbf{MH}_{*,*}(\mathcal{G})\oplus \mathbf{MH}_{*,*}(\mathcal{H})\rightarrow \mathbf{MH}_{*,*}(\mathcal{G}\sqcup\mathcal{H})
$$
is an isomorphism.
\end{proposition}
\begin{proof}
Suppose that $(\sigma_0,\dots,\sigma_k)$ is a generator of $\mathbf{MC}_{k,l}(\mathcal{G}\sqcup \mathcal{H})$, satisfying $\mathcal{L}(\sigma_0,\dots,\sigma_k)=l$. If the condition $\mathcal{L}(\sigma_0,\dots,\sigma_k)=l$ holds, it means $d(\sigma_{i-1},\sigma_i)< \infty$, $1\leq i\leq k$. Therefore, there is no path between $\mathcal{G}$ and $\mathcal{H}$, that is, $(\sigma_0,\dots,\sigma_k)$ all belong to $\mathcal{G}$ or $\mathcal{H}$. Then, $i_{\#} \bigoplus j_{\#}$ is an isomorphism. The result is completed.
\end{proof}
\begin{corollary}
Let $\mathcal{G}$ and $\mathcal{H}$ be hypergraphs. Then $\#(\mathcal{G}\sqcup \mathcal{H})=\#\mathcal{G}+\#\mathcal{H}$.
\end{corollary}
\begin{proof}
By Theorem \ref{Eluer}, we know $\sum\limits_{n,l\geq 0}(-1)^{n}\mathrm{rank}(\mathbf{MH}_{n,l}(\mathcal{G}))\cdot q^{l}=\sum\limits_{l\geq 0}\chi(\mathbf{MH}_{*,l}(\mathcal{G}))\cdot q^{l}=\#\mathcal{G}.$ Similarly, we also know $\#\mathcal{H}=\sum\limits_{l\geq 0}\chi(\mathbf{MH}_{*,l}(\mathcal{H}))\cdot q^{l}$. By Proposition \ref{disjoint}, we have $\chi(\mathbf{MH}_{*,l}(\mathcal{G})\sqcup \mathbf{MH}_{*,l}(\mathcal{H}))=\chi(\mathbf{MH}_{*,l}(\mathcal{G}))+\chi(\mathbf{MH}_{*,l}(\mathcal{H}))$. So, we have $\#(\mathcal{G}\sqcup \mathcal{H})=\#\mathcal{G}+\#\mathcal{H}$.
\end{proof}

\section{Simple magnitude homology and K\"{u}nneth theorem}\label{section:kunneth}

In this section, we will introduce the K\"{u}nneth theorem for simple magnitude homology of hypergraphs. The Cartesian product of hypergraphs has been extensively studied by various researchers since the 1960s, making it one of the most well-researched constructions in hypergraph theory \cite{imrich1967kartesisches,hellmuth2012survey,bretto2009cartesian,bretto2010factorization,bretto2006hypergraphs,winitz1981native,imrich1971schwache}.
While the Cartesian product of hypergraphs is a valuable tool in various mathematical areas, it does not possess the K\"unneth theorem for magnitude homology of hypergraphs. To prove our theorem, we will introduce magnitude simplicial sets, which can be considered as a realization of the simple magnitude chain complexes in the form of simplicial sets.
\subsection{Simple magnitude homology}
In Section \ref{section:magnitude}, we introduce the magnitude chain complex of hypergraphs by utilizing tuples that consist of hyperedges. However, directly computing this magnitude is a complex task. To address this issue, we present a simplified version which is called the simple magnitude homology, which is better suited for computational purposes. Moreover, the simple magnitude homology of hypergraphs can be viewed as a generalization of the magnitude homology of graphs. Notably, the simple magnitude homology also exhibits a K\"{u}nneth formula, further enhancing its utility.

Let $\mathcal{H}=(V,E)$ be a hypergraph. We consider the tuples of the form $(v_{0},\dots,v_{k})$ with $v_{0},v_{1},\dots,v_{k}\in V$ such that $v_{i-1}\neq v_{i}$ for $i=1,2,\dots,k$. We can also define the length of the tuple $(v_{0},\dots,v_{k})$ by
\begin{equation*}
  \mathcal{L}(v_{0},v_{1},\dots,v_{k})=d(v_{0},v_{1})+\cdots+d(v_{k-1},v_{k}).
\end{equation*}
Here, $d(v,w)$ denotes the intercrossing distance from $\{v\}$ to $\{w\}$ on the hypergraph $\mathcal{H}$. Note that we do not require that $\{v\},\{w\}$ are hyperedges in $\mathcal{H}$ here. Let $\bar{\mathcal{H}}$ be a hyperedge obtained by including all the vertices as $0$-hyperedges.
And a path from $\{v\}$ to $\{w\}$ can be regarded as path on $\bar{\mathcal{H}}$. This idea is based on the fact that the addition of vertices does not change the intercrossing distance.

Now, let $A$ be an abelian group. Let $\mathrm{MC}_{k,l}(\mathcal{H};A)$ be a free abelian group over $A$ generated by all such tuples. Similarly, we have a {\em simple maginitude chain complex} $\mathrm{MC}_{\ast,\ast}(\mathcal{H};A)$ with the differential $\partial_{l}:\mathrm{MC}_{k,l}(\mathcal{H};A)\to \mathrm{MC}_{k-1,l}(\mathcal{H};A)$ is given by $\partial_{l}=\sum\limits_{i=0}^{k}(-1)^i\partial_{i,l}$ with
\begin{equation*}
  \partial_{i,l}(v_{0},\dots,v_{k})=\left\{
                                              \begin{array}{ll}
                                                (v_{0},\dots,\hat{v_{i}},\dots,v_{k}) & \hbox{$\mathcal{L}(v_{0},\dots,\hat{v_{i}},\dots,v_{k})=l$} , \\
                                                0 & \hbox{otherwise.}
                                              \end{array}
                                            \right.
\end{equation*}
Then the {\em simple magnitude homology} is defined by
\begin{equation*}
  \mathrm{MH}_{k,l}(\mathcal{H};A)=H_{k,l}(\mathrm{MC}_{\ast,\ast}(\mathcal{H};A)),\quad k,l\geq 0.
\end{equation*}

Recall that a graph can be regarded as a hypergraph with the hyperedges given by the edges from the original graph.
It is worth noting that the simple magnitude homology of hypergraphs coincides with the magnitude homology of graphs when $\mathcal{H}=(V,E)$ is a graph.

\begin{example}
(\textbf{The simple magnitude homology of a hypergraph})
Example \ref{em:2.6} continued, we calculate the simple magnitude homology of the given hypergraph. Likewise, we have two significant variables, $k$ and $l$. Furthermore, we only consider the values of $k$ up to 2.
\begin{itemize}
  \item [(i)] When $k=0$, $l=0$, we have $(\{0\})$, $(\{1\})$,  $(\{2\})$.
  \item [(ii)] When $k=1$, $l$ can be $1$ and $2$.
  \begin{itemize}
   \item[(a)]$l=1$, tuples are $(\{0\}, \{1\})$, $(\{1\}, \{2\})$, $(\{1\}, \{0\})$, $( \{2\},\{1\})$.
  \item[(b)]$l=2$, we get $(\{0\}, \{2\})$, $(\{2\}, \{0\})$.
  \end{itemize}
  \item [(iii)] When $k=2$, $l$ is $2$, $3$ and $4$.
  \begin{itemize}
   \item[(a)]$l=2$, tuples are  $(\{0\}, \{1\}, \{0\})$, $(\{0\}, \{1\}, \{2\})$, $(\{1\}, \{2\}, \{1\})$, $(\{1\}, \{0\},$ $\{1\})$, $( \{2\},\{1\}, \{0\})$, $( \{2\},\{1\}, \{2\})$.
  \item[(b)]$l=3$, we have $(\{0\}, \{2\}, \{1\})$, $(\{1\}, \{2\}, \{0\})$, $(\{1\}, \{0\}, \{2\})$, $(\{2\}, \{0\}, \{1\})$.
  \item[(c)]$l=4$, we obtain $(\{0\}, \{2\}, \{0\})$, $(\{2\}, \{0\}, \{2\})$.
  \end{itemize}
\end{itemize}

Now, we consider the simple magnitude chain complex for $l=0, 1, 2$. The simple chain complexes are given by
$$
\cdots\longrightarrow \mathrm{MC}_{2,0}(\mathcal{H})\stackrel{\partial_0}\longrightarrow \mathrm{MC}_{1,0}(\mathcal{H})\stackrel{\partial_0}\longrightarrow \mathrm{MC}_{0,0}(\mathcal{H})\longrightarrow 0,
$$
$$
\cdots\longrightarrow \mathrm{MC}_{2,1}(\mathcal{H})\stackrel{\partial_1}\longrightarrow \mathrm{MC}_{1,1}(\mathcal{H})\stackrel{\partial_1}\longrightarrow \mathrm{MC}_{0,1}(\mathcal{H})\longrightarrow 0,
$$
and
$$
\cdots\longrightarrow \mathrm{MC}_{2,2}(\mathcal{H})\stackrel{\partial_2}\longrightarrow \mathrm{MC}_{1,2}(\mathcal{H})\stackrel{\partial_2}\longrightarrow \mathrm{MC}_{0,2}(\mathcal{H})\longrightarrow 0.
$$
Using a computation method similar to that in Example \ref{em:2.6}, we can determine the magnitude homology $\mathrm{MH}_{0,0}(\mathcal{H})=\mathbb{Z}\oplus \mathbb{Z} \oplus \mathbb{Z}$, $\mathrm{MH}_{k,0}(\mathcal{H})=0$ for $k\neq 0$, $\mathrm{MH}_{1,1}(\mathcal{H})=\mathbb{Z}\oplus \mathbb{Z} \oplus \mathbb{Z}\oplus \mathbb{Z}$, $\mathrm{MH}_{1,2}(\mathcal{H})= 0$, and so on.
\end{example}

From now on, our primary focus is to investigate the K\"{u}nneth formula for the simple magnitude homology of hypergraphs. Let $\mathcal{G}=(V_1,E_1)$ and $\mathcal{H}=(V_2,E_2)$ be hypergraphs. The \emph{Cartesian product} $\mathcal{G}\Box \mathcal{H}$ of hypergraphs has set of vertices $V_1\times V_2$ and set of hyperedges $E_1\Box E_2=\{\{x\}\times \tau |x\in V_1 ,\tau\in E_2\}\bigcup \{\sigma\times \{y\}|y\in V_2 ,\sigma\in E_1 \} $. The Cartesian product $\Box$ is associative and commutative up to isomorphism.

\begin{lemma}\label{lemma:prelemma}
Let $\mathcal{G}=(V_1,E_1)$ and $\mathcal{H}=(V_2,E_2)$ be hypergraphs.  Let $\sigma\times \tau,\sigma'\times \tau'\in E_1\Box E_2$. Suppose $\sigma\times \tau,\sigma'\times \tau'$ has a nonempty intersection. Then we have
\begin{equation*}
    \ell_{\mathcal{G}\Box \mathcal{H}}(\sigma\times \tau,\sigma'\times \tau')=\ell_{\mathcal{G}}(\sigma,\sigma')+\ell_{\mathcal{H}}(\tau,\tau').
\end{equation*}
\end{lemma}
\begin{proof}
$(i)$ When $\sigma\times \tau\subseteq\sigma'\times \tau'$ or $\sigma\times \tau\subseteq\sigma'\times \tau'$. We only prove the case $\sigma\times \tau\subseteq\sigma'\times \tau'$. It follows that $\sigma\subseteq \sigma'$ and $\tau\subseteq\tau'$. If $\sigma= \sigma'$ and $\tau=\tau'$, it is trivial. If $\sigma= \sigma'$ and $\tau\subsetneq\tau'$, one has $\ell_{\mathcal{G}\Box \mathcal{H}}(\sigma\times \tau,\sigma'\times \tau')=1/2=\ell_{\mathcal{G}}(\sigma,\sigma')+\ell_{\mathcal{H}}(\tau,\tau')$. Similarly, if $\sigma\subsetneq \sigma'$ and $\tau=\tau'$, we also have the desired equality. It is impossible that $\sigma\subsetneq \sigma'$ and $\tau\subsetneq\tau'$ since $\sigma\times \tau,\sigma'\times \tau'\in E_1\Box E_2$.

$(ii)$ When $\sigma\times \tau\not\subseteq\sigma'\times \tau'$ and $\sigma'\times \tau'\not\subseteq\sigma\times \tau$. Since  $\sigma\times \tau\in E_1\Box E_2$, at least one of $\sigma$ and $\tau$ is a $0$-hyperedge. We only consider the case $\sigma$ is a $0$-hyperedge. If $\sigma'$ is a $0$-hyperedge, we have $\tau\not\subseteq\tau'$ and $\tau'\not\subseteq\tau$. It follows that
\begin{equation*}
    \ell_{\mathcal{G}\Box \mathcal{H}}(\sigma\times \tau,\sigma'\times \tau')=1=\ell_{\mathcal{G}}(\sigma,\sigma')+\ell_{\mathcal{H}}(\tau,\tau').
\end{equation*}
If $\sigma'$ is not a $0$-hyperedge, then $\tau'$ is a $0$-hyperedge. Moreover, we have $\sigma\subsetneq \sigma'$ and $\tau'\subsetneq\tau$. Thus one has
\begin{equation*}
    \ell_{\mathcal{G}\Box \mathcal{H}}(\sigma\times \tau,\sigma'\times \tau')=1=\ell_{\mathcal{G}}(\sigma,\sigma')+\ell_{\mathcal{H}}(\tau,\tau').
\end{equation*}
The case that $\tau$ is a $0$-hyperedge is similar. The proof is completed.
\end{proof}

\begin{lemma}\label{lemma:product_equal}
Let $\mathcal{G}=(V_1,E_1)$ and $\mathcal{H}=(V_2,E_2)$ be hypergraphs. For $v_{1},v_{2}\in V_{1}$ and $w_{1},w_{2}\in V_{2}$, we have
\begin{equation*}
    d_{\mathcal{G}\Box \mathcal{H}}(v_{1}\times w_{1},v_{2}\times w_{2})=d_{\mathcal{G}}( v_{1}, v_{2})+d_{\mathcal{H}}( w_{1}, w_{2}).
\end{equation*}
\end{lemma}
\begin{proof}
$(i)$ By Proposition \ref{proposition:path2}, there is a path $\gamma=\{v_{1}\times w_{1}\}(\sigma_{1}\times \tau_{1})\cdots(\sigma_{k-1}\times \tau_{k-1})\{v_{2}\times w_{2}\}$ from $\{v_{1}\times w_{1}\}$ to $\{v_{2}\times w_{2}\}$ of length $\ell(\gamma)=d(v_{1}\times w_{1},v_{2}\times w_{2})$ and height $k=\delta(v_{1}\times w_{1},v_{2}\times w_{2})$ for hyperedges $\sigma_{1}\times \tau_{1},\dots,\sigma_{k-1}\times \tau_{k-1}\in E_{1}\Box E_{2}$. It is worth noting that the hyperedges $\sigma_{1}\times \tau_{1},\dots,\sigma_{k-1}\times \tau_{k-1}$ are of dimensional $\geq 1$.

Since $\{v_{1}\times w_{1}\}\cap(\sigma_{1}\times \tau_{1})\neq \emptyset$, we have $\{v_{1}\times w_{1}\}\subsetneq(\sigma_{1}\times \tau_{1})$. It follows that
\begin{equation*}
    \ell_{\mathcal{G}\Box \mathcal{H}}(\{v_{1}\times w_{1}\},\sigma_1\times \tau_1)=\frac{1}{2}.
\end{equation*}
Recall that one of $\sigma_{1},\tau_{1}$ is a $0$-hyperedge. So we have
\begin{equation*}
   \ell_{\mathcal{G}}(\{v_{1}\},\sigma_{1})+\ell_{\mathcal{H}}(\{w_{1}\},\tau_{1})= \ell_{\mathcal{G}\Box \mathcal{H}}(\{v_{1}\times w_{1}\},\sigma_1\times \tau_1).
\end{equation*}
Similarly, one has
\begin{equation*}
   \ell_{\mathcal{G}}(\sigma_{k-1},\{v_{2}\})+\ell_{\mathcal{H}}(\tau_{k-1},\{w_{2}\})= \ell_{\mathcal{G}\Box \mathcal{H}}(\sigma_{k-1}\times \tau_{k-1},\{v_{2}\times w_{2}\}).
\end{equation*}
Note that $\gamma_{1}=\{v_{1}\}\sigma_{1}\cdots\sigma_{k-1}\{v_{2}\}$ is a path from $\{v_{1}\}$ to $\{v_{2}\}$ and $\gamma_{2}=\{w_{1}\}\tau_{1}\cdots\tau_{k-1}\{w_{2}\}$ is a path from $\{w_1\}$ to $\{w_2\}$.
By definition and Lemma \ref{lemma:prelemma}, we have
\begin{equation*}
\begin{split}
    &\ell(\gamma)\\
    =&\ell_{\mathcal{G}\Box \mathcal{H}}(\{v_1\times w_1\},\sigma_{1}\times \tau_1)+\ell_{\mathcal{G}\Box \mathcal{H}}(\sigma_{k-1}\times \tau_{k-1},\{v_2\times w_2\})+\sum\limits_{j=1}^{k-2}\ell_{\mathcal{G}\Box \mathcal{H}}(\sigma_{j}\times \tau_{j},\sigma_{j+1}\times \tau_{j+1})\\
    =&\ell_{\mathcal{G}}(\{v_1\},\sigma_{1})+\ell_{\mathcal{G}}(\sigma_{k-1},\{v_2\})+\sum\limits_{j=1}^{k-2}\ell_{\mathcal{G}}(\sigma_{j},\sigma_{j+1})+\ell_{ \mathcal{H}}(\{w_1\},\tau_{1})+\ell_{\mathcal{H}}( \tau_{k-1}, \{w_2\})+\sum\limits_{j=1}^{k-2}\ell_{\mathcal{H}}(\tau_{j},\tau_{j+1})\\
    =&\ell_{\mathcal{G}}(\gamma_{1})+\ell_{\mathcal{H}}(\gamma_{2})\\
\end{split}
\end{equation*}
It follows that
\begin{equation*}
   d_{\mathcal{G}\Box \mathcal{H}}(v_{1}\times w_{1},v_{2}\times w_{2})=\ell(\gamma)= \ell_{\mathcal{G}}(\gamma_{1})+\ell_{\mathcal{H}}(\gamma_{2})\geq d_{\mathcal{G}}(v_{1}, v_{2})+d_{\mathcal{H}}( w_{1},w_{2}).
\end{equation*}

$(ii)$ On the other hand, by Proposition \ref{proposition:path}, suppose $\gamma_{1}=\{v_{1}\}\sigma_{1}\cdots\sigma_{k-1}\{v_{2}\}$ is a path from $\{v_{1}\}$ to $\{v_{2}\}$ of length $\ell(\gamma_{1})=d_{\mathcal{G}}( v_1, v_2)$ and height $k=\delta_{\mathcal{G}}( v_1, v_2 )$. Similarly, suppose $\gamma_{2}=\{w_1\}\tau_{1}\cdots\tau_{m-1}\{w_2\}$ is a path from $\{w_1\}$ to $\{w_2\}$ of length $\ell(\gamma_{2})=d_{\mathcal{H}}( w_1, w_2 )$ and height $m=\delta_{\mathcal{H}}(w_1, w_2)$. Then we have a path
\begin{equation*}
  \gamma= \{v_1\times w_1\}(\sigma_{1}\times w_1)\cdots (\sigma_{k-1}\times w_1)(v_2\times w_1)(v_2\times \tau_{1})\cdots(v_2\times \tau_{m-1})\{v_2\times w_2\}
\end{equation*}
from $\{v_1\times w_1\}$ to $\{v_2\times w_2\}$ of height $k+m$. Moreover, we have
\begin{equation*}
\begin{split}
    \ell(\gamma)=&\ell_{\mathcal{G}\Box \mathcal{H}}(\{v_1\times w_1\},\sigma_{1}\times \{w_1\})+\ell_{\mathcal{G}\Box \mathcal{H}}(\sigma_{k-1}\times \{w_1\},\{v_2\times w_1\})+\sum\limits_{j=1}^{k-2}\ell_{\mathcal{G}\Box \mathcal{H}}(\sigma_{j}\times \{w_1\},\sigma_{j+1}\times \{w_1\})\\
    &+\ell_{\mathcal{G}\Box \mathcal{H}}(\{v_2\times w_1\},\{v_2\}\times \tau_{1})+\ell_{\mathcal{G}\Box \mathcal{H}}(\{v_2\}\times \tau_{m-1},\{v_2\times w_2\})+\sum\limits_{j=1}^{m-2}\ell_{\mathcal{G}\Box \mathcal{H}}(\{v_2\}\times \tau_{j},\{v_2\}\times \tau_{j+1})\\
    =&\ell_{\mathcal{G}}(\{v_1\},\sigma_{1})+\ell_{\mathcal{G}}(\sigma_{k-1},\{v_2\})+\sum\limits_{j=1}^{k-2}\ell_{\mathcal{G}}(\sigma_{j},\sigma_{j+1})+\ell_{ \mathcal{H}}(\{w_1\},\tau_{1})+\ell_{\mathcal{H}}( \tau_{m-1},\{ w_2\})+\sum\limits_{j=1}^{m-2}\ell_{\mathcal{H}}(\tau_{j},\tau_{j+1})\\
    =&\ell_{\mathcal{G}}(\gamma_{1})+\ell_{\mathcal{H}}(\gamma_{2})\\
    =&d_{\mathcal{G}}( v_1, v_2)+d_{\mathcal{H}}( w_1, w_2).
\end{split}
\end{equation*}
Hence, we have
\begin{equation*}
   d_{{\mathcal{G}\Box \mathcal{H}}}(v_{1}\times w_{1},v_{2}\times w_{2})\leq
   \ell(\gamma)=d_{\mathcal{G}}(v_1, v_2)+d_{\mathcal{H}}( w_1, w_2).
\end{equation*}
This completes the proof.
\end{proof}

\begin{definition}(\textbf{The magnitude simplicial set $M_l(\mathcal{G})$})
Let $\mathcal{G}$ be a hypergraph. Then we have a based simplicial set $(M_l(\mathcal{G}),\ast)$ whose $k$-simplices are the $(k+1)$-tuples $(v_{0},v_{1},\dots,v_{k})$ with the length $l\geq 0$, which has $i$-th face map
\begin{equation*}
 d_i(v_{0},v_{1},\dots,v_{k})=\left\{
                                              \begin{array}{ll}
                                                (v_{0},\dots,v_{i-1},v_{i+1},\dots,v_{k}) & \hbox{$\mathcal{L}(v_{0},v_{1},\dots,v_{i-1},v_{i+1},\dots,v_{k})=l$,} \\
                                                \ast & \hbox{otherwise},
                                              \end{array}
                                            \right.
\end{equation*}
and the $i$-th degeneracy
\begin{equation*}
 s_i(v_{0},v_{1},\dots,v_{k})=\left\{
                                              \begin{array}{ll}
                                                (v_{0},\dots,v_{i-1},v_{i},v_{i},v_{i+1},\dots,v_{k}) & \hbox{$\mathcal{L}(v_{0},v_{1},\dots,v_{i-1},v_{i},v_{i},v_{i+1},\dots,v_{k})=l$,} \\
                                                \ast & \hbox{otherwise.}
                                              \end{array}
                                            \right.
\end{equation*}
\end{definition}

\begin{proposition}\label{proposition:simplicial}
Let $\mathcal{G},\mathcal{H}$ be hypergraphs. For $l\geq 0$, the map of pointed simplicial sets
    \begin{equation*}
       \Box:\underset{p+q=l}{\vee}M_{p}(\mathcal{G})\wedge M_{q}(\mathcal{H})\rightarrow M_l(\mathcal{G}\Box \mathcal{H})
    \end{equation*}
defined by $(v_1,\dots,v_k)\Box (w_1,\dots,w_k)=(v_1\times w_1,\dots,v_k\times w_k)$ is an isomorphism.
\end{proposition}
\begin{proof}
By Lemma \ref{lemma:product_equal}, we have
\begin{equation*}
    d_{\mathcal{G}\Box \mathcal{H}}(v_1\times w_1,\dots,v_k\times w_k)=d_{\mathcal{G}}(v_1,\dots,v_k)+d_{\mathcal{H}}(w_1,\dots,w_k).
\end{equation*}
Note that the map is simplicial. Thus the map is well defined. The isomorphism is obtained by a direct verification.
\end{proof}

The simplicial set $M_l(\mathcal{G})$ provided us with a convenient way to deal with the magnitude complex. Recall that the normalized chain complex of a simplicial complex $K$ is the chain complex $C_{\ast}(K;A)$ quotient the degenerate part. More precisely, the normalized chain complex is given by
\begin{equation*}
  N_{\ast}(K;A)=C_{\ast}(K;A)/D_{\ast}(K;A),
\end{equation*}
where, $D_{\ast}(K;A)$ is the sub chain complex of $C_{\ast}(K;A)$ generated by the degenerate elements.
Let $K,L$ be two simplicial sets. We have the Eilenberg-Zilber map
\begin{equation*}
   EZ: N_{\ast}(K)\otimes N_{\ast}(L)\to N_{\ast}(K\times L)
\end{equation*}
given by
\begin{equation*}
EZ(\sigma\otimes \tau)= \sum_{(\mu,\nu)}(-1)^{n(\mu,\nu)}\cdot(x_{\mu(0)}\times y_{\nu(0)},\dots,x_{\mu(p+q)}\times y_{\nu(p+q)}),
\end{equation*}
for non-degenerate elements $\sigma=\{x_0,\dots,x_{p}\}\in K$ and $\tau=\{y_0,\dots,y_{q}\}\in L$, where $(\mu,\nu)$ is given by $\mu(0)=\nu(0)=0$, $\mu(p+q)=p$, $\nu(p+q)=q$ and either
\begin{equation*}
\left\{
  \begin{array}{ll}
    \mu(i+1)=\mu(i) \\
    \nu(i+1)=\nu(i)+1
  \end{array}
\right. \quad\text{or}\quad
\left\{
  \begin{array}{ll}
    \mu(i+1)=\mu(i)+1 \\
    \nu(i+1)=\nu(i)
  \end{array}
\right.
\end{equation*}
for $0\leq i\leq p+q-1$. Here, $n(\mu,\nu)=\sum\limits_{0\leq i<j\leq p+q}[\mu(i+1)-\mu(i)][\nu(j+1)-\mu(i)]$\cite{maclane2012homology}.

Now, let $X,Y$ be two pointed simplicial sets. Then we have the reduced normalized chain complexes $\bar{N}_{\ast}(X),\bar{N}_{\ast}(Y)$ of $X,Y$. The Eilenberg-Zilber map of reduced version
\begin{equation}\label{equation:isomorphism}
   \overline{EZ}: \bar{N}_{\ast}(X)\otimes \bar{N}_{\ast}(Y)\to \bar{N}_{\ast}(X\wedge Y)
\end{equation}
is an chain homotopy equivalence \cite{hepworth2017categorifying}.

Moreover, by observing the definition of magnitude complexes, we find that
the reduced normalized chain complex of the pointed simplicial set $M_l(\mathcal{H})$ coincides with the corresponding magnitude complex of $\mathcal{H}$.
\begin{lemma}\label{lemma:normalized}
Let $\mathcal{H}$ be a hypergraph. Then we have
    $\bar{N}_{\ast}(M_{l}(\mathcal{H}))=\mathrm{MC}_{\ast,l}(\mathcal{H})$.
\end{lemma}

\subsection{The proof of K\"{u}nneth theorem for magnitude homology}
Now, we will show the K\"{u}nneth theorem for magnitude homology of hypergraphs with respect to the product of hypergraphs introduced in the last subsection. We give the defintion of the exterior product first.

\begin{definition}(\textbf{The exterior product})
Let $\mathcal{G}$ and $\mathcal{H}$ be hypergraphs.
We can define \emph{the exterior product} $\Box:\mathrm{MC}_{*,*}(\mathcal{G})\otimes \mathrm{MC}_{*,*}(\mathcal{H})\rightarrow \mathrm{MC}_{*,*}(\mathcal{G}\Box \mathcal{H})$ as follows.
For $k_1, k_2\geq 0$ and $l_1, l_2\geq 0$, the map
\begin{equation*}
\Box:\mathrm{MC}_{k_1,l_1}(\mathcal{G})\otimes \mathrm{MC}_{k_2,l_2}(\mathcal{H})\rightarrow \mathrm{MC}_{k_1+k_2,l_1+l_2}(\mathcal{G}\Box  \mathcal{H})
\end{equation*}
is defined by
\begin{equation*}
(v_0,\dots,v_{k_1})\square(w_0,\dots,w_{k_2})= \sum_{(\mu,\nu)}(-1)^{n(\mu,\nu)}\cdot(v_{\mu(0)}\Box w_{\nu(0)},\dots,v_{\mu(k_1+k_2)}\Box w_{\nu(k_1+k_2)}),
\end{equation*}
where $(\mu,\nu)$ is given by $\mu(0)=\nu(0)=0$, $\mu(k_1+k_2)=k_{1}$, $\nu(k_1+k_2)=k_{2}$ and either
\begin{equation*}
\left\{
  \begin{array}{ll}
    \mu(i+1)=\mu(i) \\
    \nu(i+1)=\nu(i)+1
  \end{array}
\right. \quad\text{or}\quad
\left\{
  \begin{array}{ll}
    \mu(i+1)=\mu(i)+1 \\
    \nu(i+1)=\nu(i)
  \end{array}
\right.
\end{equation*}
for $0\leq i\leq k_1+k_2-1$. Here, $n(\mu,\nu)=\sum\limits_{0\leq i<j\leq k_{1}+k_{2}}[\mu(i+1)-\mu(i)][\nu(j+1)-\mu(i)]$.
\end{definition}

\begin{theorem}\emph{(\cite[Theorem 3B.5]{hatcher2002algebraic})}\label{thm3}
Let $R$ be a principal ideal domain, and let $C,C'$ be chain complexes of free $R$-modules. Then there is a natural exact sequence
\begin{equation*}
  0\rightarrow \bigoplus_{p+q=n}H_{p}(C)\otimes H_{q}(C')\rightarrow H_{n}(C\otimes C')\rightarrow  \bigoplus_{p+q=n}\mathrm{Tor}_{R}(H_{p}(C),H_{q-1}(C'))\rightarrow 0.
\end{equation*}
\end{theorem}

\begin{theorem}
\textbf{\emph{(The K\"{u}nneth theorem for simple magnitude homology of hypergraphs)}}\label{kunneth}
Let $\mathcal{G}$ and $\mathcal{H}$ be hypergraphs. By the exterior product, we have a natural short exact sequence
\begin{equation*}
\begin{split}
      0\rightarrow \bigoplus\limits_{p+q=n}\mathrm{MH}_{p,*}(\mathcal{G})\otimes \mathrm{MH}_{q,*}(\mathcal{H})&\stackrel{\square}{\rightarrow}\mathrm{MH}_{n,*}(\mathcal{G}\Box\mathcal{H})\\
      &\rightarrow \bigoplus\limits_{p+q=n}\mathrm{Tor}(\mathrm{MH}_{p,*}(\mathcal{G}),\mathrm{MH}_{q-1,*}(\mathcal{H}))\rightarrow 0.
\end{split}
\end{equation*}
\end{theorem}

\begin{proof}
By Eq. \ref{equation:isomorphism} and Proposition \ref{proposition:simplicial}, we have
\begin{equation*}
  \begin{split}
\bigoplus\limits_{l_1+l_2=l} \bar{N}_{\ast} (M_{l_1}(\mathcal{G}))\otimes \bar{N}_{\ast}(M_{l_2}(\mathcal{H}))&\stackrel{\overline{EZ}}{\longrightarrow}\bigoplus\limits_{l_1+l_2=l} \bar{N}_{\ast}(M_{l_1}(\mathcal{G})\wedge M_{l_2}(\mathcal{H}))\\
&\stackrel{=}\longrightarrow\bar{N}_{\ast}(\bigvee\limits_{l_1+l_2=l}M_{l_1}(\mathcal{G})\wedge M_{l_2}(\mathcal{H}))\\
      &\stackrel{\bar{N}(\Box)}\longrightarrow \bar{N}_{\ast}(M_l(\mathcal{G}\Box\mathcal{H})).
  \end{split}
\end{equation*}
By Lemma \ref{lemma:normalized}, the above quasi-isomorphism can be reduced to a quasi-isomorphism
\begin{equation*}
\bigoplus\limits_{l_1+l_2=l} \mathrm{MC}_{\ast,l_{1}} (\mathcal{G})\otimes  \mathrm{MC}_{\ast,l_{2}}(\mathcal{H})\stackrel{\simeq}{\longrightarrow}  \mathrm{MC}_{\ast,l}(\mathcal{G}\Box\mathcal{H}),
\end{equation*}
which is exactly the exterior product. Applying Theorem \ref{thm3} to the magnitude complex with respect to the first index of the magnitude complexes, we obtain  natural exact sequence
\begin{equation*}
\begin{split}
      0\rightarrow \bigoplus\limits_{\substack{p+q=n\\l_1+l_2=l}}\mathrm{MH}_{p,l_{1}}(\mathcal{G})\otimes \mathrm{MH}_{q,l_{2}}(\mathcal{H})&\stackrel{\square}{\rightarrow}\bigoplus\limits_{l_1+l_2=l} H_{n}(\mathrm{MC}_{\ast,l_{1}} (\mathcal{G})\otimes  \mathrm{MC}_{\ast,l_{2}}(\mathcal{H}))\\
      &\rightarrow \bigoplus\limits_{\substack{p+q=n\\l_1+l_2=l}}\mathrm{Tor}(\mathrm{MH}_{p,l_{1}}(\mathcal{G}),\mathrm{MH}_{q-1,l_{2}}(\mathcal{H}))\rightarrow 0.
\end{split}
\end{equation*}
Combining with the quasi-isomorphism before, we obtain the desired result.
\end{proof}

\medskip
\noindent {\bf Acknowledgement}.
The authors are supported by Natural Science Foundation of China (NSFC grant no. 11971144) and the start-up research fund from BIMSA.
%
%
%
%
%
\bibliographystyle{plain}
\bibliography{mag_refs}
\bigskip

Wanying Bi

Address: $^1$School of Mathematical Sciences, Hebei Normal University, 050024, China.

$^2$Yanqi Lake Beijing Institute of Mathematical Sciences and Applications, 101408, China.

e-mail:wanyingbi1015@163.com

\medskip

Jiangyan Li

Address: Yanqi Lake Beijing Institute of Mathematical Sciences and Applications, 101408, China.

e-mail: jingyanli@bimsa.cn

\medskip

Jie Wu

Address: Yanqi Lake Beijing Institute of Mathematical Sciences and Applications, 101408, China.

e-mail: wujie@bimsa.cn

\medskip

\end{CJK*}
 \end{document}